\documentclass[twoside,reqno,12pt]{amsart}
\usepackage{times}
\usepackage[leqno]{amsmath}
\usepackage{amssymb,amsthm}
\usepackage{amsrefs}
\usepackage{amscd}
\usepackage{color}
\usepackage{enumerate}
\usepackage{cite}
\usepackage{mathrsfs}

\allowdisplaybreaks

\makeatletter
\newcommand{\leqnomode}{\tagsleft@true\let\veqno\@@leqno}
\newcommand{\reqnomode}{\tagsleft@false\let\veqno\@@eqno}
\makeatother

\usepackage{calc}
\newtheorem{thm}{Theorem}[section]
\newtheorem{lem}[thm]{Lemma}
\newtheorem{prop}[thm]{Proposition}

\theoremstyle{definition}
\newtheorem{defn}[thm]{Definition}

\theoremstyle{remark}
\newtheorem*{rem}{Remark}
\newtheorem*{ex}{Example}
 
\numberwithin{equation}{section}

\newcommand{\N}{{\mathbb N}}
\newcommand{\R}{{\mathbb R}}
\newcommand{\Rp}{\R_+}
\newcommand{\RR}{\R_+^{1+d}}
\newcommand{\oRR}{\overline\R_+^{1+d}}
\newcommand{\C}{{\mathbb C}}
\newcommand{\Z}{{\mathbb Z}}

 \newcommand{\p}{\partial}
\newcommand{\cA}{{\mathcal A}}
\newcommand{\cB}{{\mathcal B}}
\newcommand{\cC}{{\mathscr C}}
\newcommand{\cL}{{\mathscr L}}
\newcommand{\cH}{{\mathscr H}}
\newcommand{\cK}{{\mathscr K}}
\newcommand{\cS}{{\mathscr S}}
\newcommand{\cl}{{\textnormal{cl}}}
\newcommand{\cc}{{\textnormal{c}}}

\newcommand{\loc}{{\textnormal{loc}}}

\newcommand{\f}{\frac}
\newcommand{\ve}{\varepsilon}

\newcommand{\As}{\underline{\textup{As}}}

\newcommand{\ler}{\langle\eta\rangle}

\newcommand{\ii}{\textup{i}}
\newcommand{\ee}{\textup{e}}
\newcommand{\dd}{\textup{d}}
\renewcommand{\dbar}{\dd\hspace*{-0.14em}\bar{}\hspace*{0.18em}}

\begin{document}


\title
{Hyperbolic problems with totally characteristic boundary}
\author[Z.-P.~Ruan]{Zhuoping Ruan}
\address{Department of Mathematics and Institute of Mathematical Sciences,
  Nanjing University, Nanjing, 210093, China}
\email{zhuopingruan@nju.edu.cn}

\author[I.~Witt]{Ingo Witt}
\address{Mathematical Institute, University of G\"{o}ttingen,
  Bunsenstr.~3-5, 37073 G\"ottingen, Germany}
\email{iwitt@uni-math.gwdg.de}

\date{\today}

\keywords{First-order hyperbolic systems, totally characteristic
  boundary, asymptotic expansions of conormal type, discrete
  asymptotic types, Sobolev spaces with asymptotics, calculus of
  cone-degenerate pseudo-differential operators, holomorphic conormal
  symbols, singular analysis}

\subjclass[2010]{Primary: 35L04; Secondary: 35L80, 35S05}

\thanks{The first author was supported by NSFC grant 11771206. Part of
  the work was done while she was visiting the University of
  G\"ottingen.}


\begin{abstract}
We study first-order symmetrizable hyperbolic $N\times N$ systems in a
spacetime cylinder whose lateral boundary is totally
characteristic. In local coordinates near the boundary at $x=0$,
these systems take the form
\[
  \p_t u + \cA(t,x,y,xD_x,D_y) u = f(t,x,y), \quad
  (t,x,y)\in(0,T)\times\R_+\times\R^d,
\]
where $\cA(t,x,y,xD_x,D_y)$ is a first-order differential operator
with coefficients smooth up to $x=0$
and the derivative with respect to $x$ appears in the
combination $xD_x$. No boundary conditions are required in such a
situation and corresponding initial-boundary value problems are
effectively Cauchy problems.

We introduce a certain scale of Sobolev spaces with asymptotics and
show that the Cauchy problem for the operator $\p_t +
\cA(t,x,y,xD_x,D_y)$ is well-posed in that scale. More specifically,
solutions $u$ exhibit formal asymptotic expansions of the form
\[
   u(t,x,y) \sim \sum_{(p,k)} \f {(-1)^k} {k!}x^{-p} \log^k \!x \,
   u_{pk}(t,y) \quad \text{as $x\to+0$}
\]
where $(p,k)\in\C\times\N_0$ and $\Re p\to-\infty$ as
$|p|\to\infty$, provided that the right-hand side $f$ and the
initial data $u|_{t=0}$ admit asymptotic expansions as $x \to +0$ of a
similar form, with the singular exponents~$p$ and their multiplicities
unchanged. In fact, the coefficient $u_{pk}$ are, in general, not
regular enough to write the terms appearing in the asymptotic
expansions as tensor products. This circumstance requires an
additional analysis of the function spaces. In addition, we
demonstrate that the coefficients~$u_{pk}$ solve certain explicitly
known first-order symmetrizable hyperbolic systems in the lateral
boundary.

Especially, it follows that the Cauchy problem for the operator
$\partial_t+\cA(t,x,y,xD_x,D_y)$ is well-posed in the scale of
standard Sobolev spaces $H^s((0,T)\times\R_+^{1+d})$.

\end{abstract}

\maketitle



\section{Introduction}



Due to their importance in the physical and engineering sciences, the
investigation of hyperbolic initial-boundary problems has a
long-standing history. Depending on the hyperbolic differential
operators under study, the main questions concern the correct number
and kind of boundary conditions to be imposed and well-posedness of
the resulting initial-boundary problems in suitable scales of function
spaces. See \textsc{Benzoni-Gavage and Serre} \cite{BGS2007} for a
recent account. In case of a non-characteristic boundary, it is known
that the weak Lopatinskii condition is necessary for well-posedness,
while the uniform Lopatinskii condition has been shown by
\textsc{Lopatinskii} \cite{Lop1970}, \textsc{Kreiss} \cite{Kre1970},
and \textsc{Sakamoto} \cite{Sak1982} to be necessary and sufficient in
order to obtain the strongest possible regularity results, comparable
to those one has for the pure Cauchy problem. See also
\textsc{Chazarain and Piriou} \cite{CP1982}. The understanding of the
characteristic case is considerably less complete. There are many
works contributing to the uniformly characteristic case, especially
for first-order systems when the differential operators under study
are symmetric hyperbolic, see e.g.~\textsc{Majda and Osher}
\cite{MO1975}, \textsc{Ohkubo} \cite{Ohk1981}, \textsc{Rauch}
\cite{Rau1985}, or \textsc{Secchi} \cite{Sec1996}. In these results,
one often has more regularity in directions tangent to the boundary
than in directions transverse to it. By contrast, for a totally
characteristic boundary, one has the same regularity in all
directions, as observed already by \textsc{Sakamoto} \cite{Sak1989}.

This observation was our point of departure. We investigate
symmetrizable hyperbolic first-order differential systems in
space-time cylinders $(0,T)\times\Omega$, where $\Omega\subseteq\R^n$
is a $\cC^\infty$ domain (or, more general, $\overline\Omega$ is a
$\cC^\infty$ manifold with non-empty boundary) and where the lateral
boundary $(0,T)\times\partial\Omega$ is totally characteristic. The
main technical innovation is to regard the differential operators
under study as cone-degenerate with respect to the spatial variables,
which in turn is possible due to the totally characteristic
boundary. As a consequence, using a suitable calculus for
cone-degenerate pseudodifferential operators (detailed below), we
construct symmetrizers for the systems under consideration and, as a
result of the symmetrization process, are able to establish
well-posedness in so-called Sobolev spaces
$H_{P,\theta}^{s,\delta}(\overline\Omega)$ with asymptotics (also
detailed below). Here, the asymptotics alluded to are discrete
conormal asymptotics, given by an asymptotic type $P$. Special cases
include the standard Sobolev spaces $H^s(\Omega)$ and
$H_0^s(\overline\Omega)$. Another relevant case is when the function
spaces carry no asymptotic information at all, i.e., when
$\theta=0$. Then the asymptotic type~$P$ is redundant and
$H_{P,0}^{s,\delta}(\overline\Omega) = \cK^{s,\delta}(\Omega)$ is a
weighted Sobolev space.

It turns out that in the situation considered no boundary conditions
are required. One main contribution of this paper is the revelation
that the boundary traces of the solutions themselves satisfy
hyperbolic differential equations in the lateral boundary
$(0,T)\times\partial\Omega$. In particular, these boundary traces can
be determined ahead of determining the solutions.

There is a long-running program of investigating conormal asymptotic
expansions of solutions to elliptic partial differential equations as
an integral part of the structure, initiated by different people, see
e.g. \textsc{Melrose} \cite{Mel1993} or \textsc{Rempel and
  Schulze}~\cite{RS1989}. Recently, there have been attempts to extend
this program to include hyperbolic partial differential equations, see
e.g.~\textsc{Hintz and Vasy} \cite{HV2020}.  This paper can also be
seen as a contribution in this direction.


\subsection{Formulation of the problem and main results}

In this paper, we investigate well-posedness of the Cauchy
problem for first-order hyperbolic systems with totally characteristic
boundary. More specifically, we consider the Cauchy problem for
$N\times N$ systems
\begin{equation}\label{ibvp}
\left\{ \enspace
\begin{aligned}
  \partial_t u + \cA(t,\varpi,D_\varpi) u
  &=
  f(t,\varpi), \quad (t,\varpi)\in (0,T)\times \Omega, \\
  u\bigr|_{t=0} &= u_0(\varpi),
\end{aligned}
\right.
\end{equation}
where $\cA\in\cC^\infty([0,T];\operatorname{Diff}^1(\overline\Omega;
\C^N))$, $\overline\Omega$ is $\cC^\infty$ manifold with non-empty
boundary $\partial\Omega$, and
$\Omega=\overline\Omega\setminus\partial\Omega$. Our standing
assumptions are that the differential operator \leqnomode
\begin{equation}
  \text{$\cL=\partial_t + \cA(t,\varpi,D_\varpi)$ is
    symmetrizable hyperbolic} \tag{\textbf{A1}}
\end{equation}
and that the lateral boundary
\begin{equation}
  \text{$(0,T)\times \partial\Omega$ is totally characteristic for
    $\cL$.} \tag{\textbf{A2}}
\end{equation}
The latter condition means that
$\sigma_\psi^1(\cA)(t,\varpi,\nu(\varpi))=0$ for $(t,\varpi)\in
[0,T]\times \partial\Omega$, where $\sigma_\psi^1(\cA)$ is the
principal symbol of $\cA$ and $\nu(\varpi) \in T_\varpi^*
\overline\Omega$ is conormal with respect to the boundary
$\partial\Omega$ (i.e., $\nu(\varpi)\bigr|_{T_\varpi
  (\partial\Omega)}=0$). \textit{A first observation is that the
  characteristic curves of $\cL$ stay inside the lateral boundary when
  they started out there. In particular, they are tangent to the
  boundary.} Consequently, no boundary conditions are required in
order to solve Eq.~\eqref{ibvp}. Besides, there is no need for a
Lopatinskii condition or an replacement of it in one or the other
form.

\reqnomode


\subsubsection{The result for standard Sobolev spaces}

We begin with describing the result for the standard Sobolev spaces
$H^s(\Omega)$, where $\Omega\subseteq\R^n$ is a $\cC^\infty$ domain and
$s\geq0$. This case deserves special attention for two reasons:
firstly, the proof here is considerably simpler than in the general
case, secondly, it helps to develop some additional intuition for the
problems studied later.

\begin{thm}\label{thm1}
Suppose that the differential operator $\partial_t+\cA$ in\/
\textup{Eq.~\eqref{ibvp}} has coefficients in
$\cC_b^\infty([0,T]\times\overline\Omega; \operatorname{Mat}_{N\times
  N}(\C))$ and that it is symmetrizable hyperbolic uniformly in
$(t,\varpi)\in[0,T]\times\overline\Omega$. Let $u_0\in
H^{s+\sigma}(\Omega;\C^N)$ and $f\in \bigcap_{r=0}^\sigma
W^{r,1}((0,T);$ $H^{s-r+\sigma}(\Omega;\C^N))$ for some $s\geq0$,
$\sigma\in\N_0$. Then\/ \textup{Eq.~\eqref{ibvp}} possesses a unique
solution
\[
  u\in\bigcap_{r=0}^\sigma \cC^r([0,T];H^{s-r+\sigma}(\Omega;\C^N)).
\]
In addition, the boundary traces
\[
  \gamma_\ell u=\frac1{\ell!}\,\frac{\partial^\ell
    u}{\partial\nu^\ell}\Bigr|_{(0,T)\times\partial\Omega}\in
  \bigcap_{\substack{r\leq \sigma,\\\ell+r<s+\sigma-1/2}}\cC^r([0,T];
  H^{s-\ell-r+\sigma-1/2}(\partial\Omega;\C^N))
\]
for $\ell\in\N_0$, $\ell <s+\sigma-1/2$ \textup{(}defined by extending
$\nu$ to a $\cC^\infty$ vector field in a neighborhood of
$\partial\Omega$\textup{)} are uniquely determined as solutions to
certain hyperbolic Cauchy problems in $(0,T)\times\partial\Omega$.
\end{thm}

See \eqref{asymp_terms}, \eqref{trace1} for the explicit form of the
hyperbolic Cauchy problems in $(0,T)\times\partial\Omega$ governing
the boundary traces~$\gamma_\ell u$. Especially, when $\gamma_\ell
u_0=0$ and $\gamma_\ell f=0$ for $\ell\in\N_0$, $\ell <s+\sigma-1/2$,
then it follows that $\gamma_\ell u=0$ for all those
$\ell$. Consequently, Theorem~\ref{thm1} remains valid when the
Sobolev spaces $H^s(\Omega)$, where $s\geq0$, are replaced with
$H_0^s(\overline\Omega)$.

\begin{rem}
The latter observation is actually one of the guiding principles in
what follows. Start with the standard Sobolev spaces $H^s(\Omega)$,
remove the asymptotic terms arising from a Taylor series expansion at
the boundary $\partial\Omega$ to arrive at the spaces
$H_0^s(\overline\Omega)$, then adjust the reference conormal order
from $0$ to $\delta$ (see, e.g., Lemma~\ref{lem_310}) and affix
asymptotic terms once again, now possibly of a different asymptotic
type $P$. This yields the function spaces
$H_{P,\theta}^{s,\delta}(\overline\Omega)$ in which well-posedness for
Eq.~\eqref{ibvp} will be shown to hold as well.
\end{rem}

\begin{rem}
The regularity of the boundary traces $\gamma_\ell u$ results from the
trace theorems for the standard Sobolev spaces. It is a half an order
less than what is obtained in case of a non-characteristic boundary
when the uniform Lopatinskii condition hold and the usual
compatibility conditions between initial and boundary data are
satisfied.
\end{rem}


\subsubsection{The local problem}\label{rlp}

Most of the effort in this paper is put into the local situation,
where $\overline\Omega=\overline\R_+^{\,n}$ is a closed half-space. We
state the result in this situation next.

We assume that the coefficients of the differential operator $\cA$
belong to $\cC_b^\infty([0,T]\times\overline{\R}_+^{\,n};$ \linebreak
$\operatorname{Mat}_{N\times N}(\C))$. We further set $n=1+d$ and
write the spatial coordinates as
$(x,y)\in\overline\R_+\times\R^d$. Then the Cauchy problem to be
investigated becomes
\begin{equation}\label{ibvp2}
\left\{ \enspace
\begin{aligned}
  \partial_t u + x\,A(t,x,y)\partial_x u +
  \sum_{j=1}^d A_j(t,x,y) \partial_j u  + B(t,x,y)u &=
  f(t,x,y), \\
  u\bigr|_{t=0} &= u_0(x,y),
\end{aligned}
\right.
\end{equation}
where $(t,x,y)\in (0,T)\times \RR$. Here,
$\partial_j=\partial/\partial y_j$,
$xA,A_j,B\in\cC_b^\infty([0,T]\times\oRR;\operatorname{Mat}_{N\times
  N}(\C))$, and $A$ is $\cC^\infty$ up to $x=0$. Further, we assume that the
linear differential operator $\cL=\partial_t + xA\partial_x +
\sum_{j=1}^d A_j\partial_j +B$ is symmetrizable hyperbolic uniformly
in $(t,x,y)\in[0,T]\times\oRR$.

\smallskip

Having concrete applications in mind, apart from asymptotics resulting
from a Taylor series expansion at $x=0$ as in Theorem~\ref{thm1},
\begin{equation}\label{taylor}
  u(t,x,y) \sim \sum_{\ell\in\N_0} \frac{x^\ell}{\ell!}\, u_\ell(t,y)
  \quad \text{as $x\to+0$,}
\end{equation}
we consider more general asymptotics of
the form
\begin{equation}\label{asp1}
  u(t,x,y) \sim \sum_{(p, k)} \frac{(-1)^k}{k!}\,x^{-p} \log^k\! x\,
  u_{p k}(t,y) \quad \text{as $x\to+0$,}
\end{equation}
where $(p,k) \in \C\times\N_0$ with $\Re p\to -\infty$ as $|p|\to
\infty$. (The precise conditions are stated in
Definition~\ref{Def3.1}.) Such asymptotics arise in many applications,
both linear and nonlinear. The choice of the exponent $-p$ (in place
of~$p$) and the appearance of the factor $(-1)^k/k!$ is related to the
normalization of the Mellin transform (see Appendix~\ref{a1}) and simplifies
certain formulas later on, especially \eqref{opmh}, \eqref{gamma_au}. We shall
denote the uniquely determined coefficients $u_{pk}$ in those
asymptotic expansions by $\gamma_{pk}u$ and \textit{regard these coefficients
as boundary traces as before}.

Incorporating the asymptotic information provided by \eqref{asp1} into
function spaces $H_{P,\theta}^{s,\delta}(\oRR)$ \linebreak (neglecting the
dependence on $t$ at this point, see Section~\ref{SSWA}), the main
result of this paper is as follows: 

\begin{thm}\label{thm2}
Let $s\in\R$, $\sigma\in\N_0$, $P\in\underline{\textup{As}}^\delta$,
and $\theta_0\geq\dotsc\geq\theta_\sigma\geq0$.  Under the assumptions
stated above, given $u_0\in H_{P,\theta_0}^{s+\sigma,\delta}(\oRR;\C^N)$
and $f\in \bigcap_{r=0}^\sigma W^{r,1}((0,T);$
$H_{P,\theta_r}^{s-r+\sigma,\delta}(\oRR;\C^N))$,
\textup{Eq.~\eqref{ibvp}} possesses a unique solution
\[
  u\in
  \bigcap_{r=0}^\sigma\cC^r([0,T];H_{P,\theta_r}^{s-r+\sigma,\delta}(\oRR;\C^N)).
\]
In addition, for all $(p,k)\in P$ with $\Re p>1/2-\delta-\theta_0$,
$\gamma_{pk}u$ solves the Cauchy problem
\begin{equation}\label{asymp_terms}
\left\{ \enspace
\begin{aligned}
  &\partial_t(\gamma_{pk}u) + \sum_{j=1}^d
  A_j(t,0,y)\partial_j(\gamma_{pk}u) +\left(- p
  A(t,0,y)+B(t,0,y)\right)\gamma_{pk}u \\
  & \hspace*{9cm}= (\gamma_{pk}f)(t,y) + R_{pk}[u](t,y), \\
  & (\gamma_{pk}u)\bigr|_{t=0} = (\gamma_{pk}u_0)(y),
\end{aligned}
\right.
\end{equation}
where the expression $R_{pk}[u]$ appearing in the right-hand side of
\eqref{asymp_terms} stands for a term that depends linearly on
$\gamma_{ql}u$ for $q - p\in \N_0$ and $l>k$ if $q=p$.
\end{thm}

Notice that Eq.~\eqref{asymp_terms} is a hyperbolic Cauchy problem for
$\gamma_{pk}u$ in $(0,T)\times\R^d$. It follows that the coefficients
$\gamma_{pk}u$ in the asymptotic expansion \eqref{asp1} can be
successively computed and are uniquely determined by the corresponding
coefficients of the initial data $u_0$ and the right-hand side
$f$.

\begin{rem}
Theorem~\ref{thm1} is a special case of Theorem~\ref{thm2} when
$\Omega=\RR$. Here, $s\geq0$, $\delta=0$, $P=P_0$, and
$\theta_r=s-r+\sigma$ for $0\leq r\leq \sigma$ (see the example after
Definition~\ref{Def3.1}).
\end{rem}

\begin{rem}
Strictly speaking, we only deal with constant discrete asymptotics. It
is likely that similar results also hold for continuous asymptotics as
well as for variable discrete asymptotics (see, e.g.,
\cite{KM2016,HMST2015,RS1989} for elliptic problems).
\end{rem}


\subsubsection{Further results}

The situation described in Section~\ref{rlp} is invariant under
coordinate changes for manifolds with boundary. A proof will appear in
\cite{LRW2022}. Hence, one also has the function spaces
$H_{P,\theta,\loc}^{s,\delta}(\overline\Omega)$ when $\overline\Omega$
is a $\cC^\infty$ manifold with boundary. In local coordinates,
elements of $H_{P,\theta,\loc}^{s,\delta}(\overline\Omega)$ belong
(locally) either to $H^s(\R^n)$ for inner charts or to
$H_{P,\theta}^{s,\delta}(\overline\R_+^n)$ for boundary charts.

\begin{thm}\label{thm3}
Let $s\geq0$, $\sigma\in\N_0$, $P\in\As^\delta$, and $\theta_0\geq
\dotsc \geq \theta_\sigma\geq0$. Furthermore, let $u_0\in
H_{P,\theta_0,\loc}^{s+\sigma,\delta}(\overline\Omega;\C^n)$ and $f\in
\bigcap_{r=0}^\sigma
W^{r,1}((0,T);H_{P,\theta_r,\loc}^{s-r+\sigma,\delta}(\overline\Omega;\C^n))$. Then\/
\textup{Eq.~\eqref{ibvp}} possesses a unique solution
\[
  u\in \bigcap_{r=0}^\sigma
  \cC^r([0,T];H_{P,\theta_r,\loc}^{s-r+\sigma,\delta}(\overline\Omega;\C^n)).
\]
Moreover, the boundary traces $\gamma_{pk}u$ for $(p,k)\in P$ with
$\Re p>1/2-\delta-\theta_0$ can be successively computed as before by
solving hyperbolic Cauchy problems in the lateral boundary
$(0,T)\times\partial\Omega$.
\end{thm}


\subsection{Outline of the argument}

In the sequel, we \textit{fix a $\delta\in\R$ as reference conormal
  order.} The weighted $L^2$ space $\cK^{0,\delta}(\RR)$ (see
Definition~\ref{Def3.6}) will be our basic Hilbert space, replacing
the space $L^2(\RR)=\cK^{0,0}(\RR)$. Denote by $\|\;\|$ the norm and
by $\langle\;,\,\rangle$ the inner product in $\cK^{0,\delta}(\RR)$.

One basic problem is to define the asymptotic terms
appearing in asymptotic expansions like
\begin{equation}\label{asp2}
  v(x,y) \sim \sum_{(p,k)\in P} \frac{(-1)^k}{k!}\,x^{-p}\log^k\!x\,
  v_{pk}(y) \quad \text{as $x\to +0$}
\end{equation}
appropriately (see \eqref{asp1}). The asymptotic expansion
\eqref{asp2} is with respect to an increasing flatness as $x\to+0$,
where the term $(-1)^k/k!\,x^{-p}\log^k\!x\, v_{pk}(y)$ has conormal
order $1/2-\Re p-0$. For $v\in H_{P,\theta}^{s,\delta}(\oRR)$, this
asymptotic expansion breaks off at conormal order $\delta+\theta$ so
that effectively only finitely many terms in the right-hand side of
\eqref{asp2} have to be taken into account. Nonetheless, as $v\in
H_{P,\theta}^{s,\delta}(\oRR)$ implies that both $\Re p<1/2-\delta$
and $v_{pk} \in H^{s+\Re p+\delta-1/2,\langle k\rangle}(\R^d)$, as a
rule we have that $(-1)^k/k!\,x^{-p}\log^k\!x\, v_{pk}(y)\notin
H_{P,\theta}^{s,\delta}(\oRR)$ near $x=0$. The correct form of the
asymptotic term is given by $\Gamma_{pk}v_{pk}$, where
\begin{equation}\label{asp4}
  (\Gamma_{pk}w)(x,y) = \frac{(-1)^k}{k!}\,\mathcal F_{\eta\to
  y}^{-1}\left\{\varphi(x\ler)\hat{w}(y) \right\} x^{-p}\log^k\! x
\end{equation}
and $\varphi$ is a cut-off function. (See Section~\ref{notat} for the
notation used.) Especially, $\Gamma_{pk}w\in \cC^\infty(\RR)$ for
$w\in\cS'(\R^d)$, $\Gamma_{pk}w\bigr|_{x=0} = w$, and $\Gamma_{pk}w$
is supported for $x\lesssim1$. The decisive property which makes the
approach work, however, is
\[
\left\{ \enspace  
\begin{aligned}
  & \Gamma_{pk}w \in \bigcap_{\ve>0} \cK^{s+\ve,1/2-\Re p-\ve}(\RR),
  \\ & \Gamma_{pk}w -\frac{(-1)^k}{k!}\,\varphi(x)x^{-p}\,\log^k\!x \,
  w(y) \in \bigcap_{\ve>0} \cK^{s-\ve,1/2-\Re p+\ve}(\RR)
\end{aligned}
\right.
\]
provided that $w\in H^{s,\langle k\rangle}(\R^d)$ (see Lemma~\ref{Lem3.8}).

\medskip

We construct a calculus $\Psi_\cc^\infty(\oRR) = \bigcup_{\mu\in\R}
\Psi_\cc^\mu(\oRR)$ of cone-degenerate pseudodifferential operators on
the half-space $\RR$, where $\Psi_\cc^\mu(\oRR)\subset
\Psi_\cl^\mu(\RR)$ and the pseudodifferential operators contained
exhibit a prearranged behavior as $x\to +0$. The basic idea is taken
from Schulze~\cite{Sch1991,Sch1998}. In particular, near $x=0$, the
operators $A\in \Psi_\cc^\mu(\oRR)$ are to the leading order of the
form
\[
  A = \varphi(x) \operatorname{op}_M(h) \varphi_0(x),
\]
where $h(z) = h(y,z,D_y)$ is an entire family of pseudodifferential
operators in $\Psi_\cl^\mu(\R^d)$ subject to further conditions,
$\operatorname{op}_M(h) = M_{z\to x}^{-1} h(z) M$ with $M$ being the
Mellin transform, and $\varphi,\varphi_0$ are cut-off
functions. Compared to the cone calculus of Schulze, where the
coefficients $v_{pk}$ would be in $\cS(\R^d)$ in the situation
considered here, we now had to show that operators $A\in
\Psi_\cc^\mu(\oRR)$ act in an appropriate way on the asymptotic terms
given by \eqref{asp4}. Indeed, it holds that
\begin{equation}\label{opmh}
  \varphi \operatorname{op}_M(h)\Gamma_{pk}w - \sum_{r=0}^k
  \frac1{r!}\, \Gamma_{p,k-r}\left[\partial_z^r h(p) \right] \in
  \bigcap_{\ve>0} \cK^{s-\ve,1/2-\Re p+\ve}(\RR)
\end{equation}
provided that $w\in H^{s,\langle k\rangle}(\R^d)$.

The operator $\cA(t,x,y,xD_x,D_y)$ from Eq.~\eqref{ibvp2} belongs to
$\cC^\infty([0,T];\Psi_\cc^1(\oRR;\C^N))$. Furthermore, a symbolic
symmetrizer for the hyperbolic operator $\cL = \partial_t +
\cA(t,x,y,xD_x,D_y)$ is indeed a symmetrizer $b(t,x,y,\tilde\xi,\eta)$
for the compressed principal symbol
$\tilde\sigma_\psi^1(\cA)(t,x,y,\tilde\xi,\eta)$ of $\cA$. Using
G{\aa}rding's inequality in a routine way yields an operator $\cB\in
\cC^\infty([0,T];\Psi_\cc^0(\oRR;\C^N))$ with
$\tilde\sigma_\psi^0(\cB)=b$ such that
\begin{itemize}
\item $\cB = \cB^*\geq c\textup{I}$ for some
  $c>0$,
\item $\Re(\cB\cA) \in \cC^\infty([0,T];\Psi_\cc^0(\oRR;\C^N))$, i.e.,
  $\Re\tilde\sigma_\psi^1(\cB\cA)=0$.
\end{itemize}
Together with the fact that integration by parts produces no boundary terms, i.e., 
\begin{itemize}
\item $\left\langle\cB\cA u,v\right\rangle = \left\langle
  u,(\cB\cA)^*v\right\rangle$ holds for $u,v\in\cC([0,T];\cK^{1,\delta}(\RR))$,
\end{itemize}
one produces, for $u\in\cC^0([0,T];\cK^{1,\delta}(\RR)) \cap
\cC^1([0,T];\cK^{0,\delta}(\RR))$, the \textit{basic energy
  inequality\/}
\begin{equation}\label{energy}
  \max_{0\leq t\leq T} \|u(t)\| \lesssim \|u(0)\| + \int_0^T
  \|\partial_t u(t) + \cA(t)u(t)\|\,\dd t
\end{equation}
in a standard manner. Once the calculus of cone-degenerate
pseudodifferential operators mentioned above is established, this
essentially means that one can treat the Cauchy problem \eqref{ibvp2}
in $(0,T)\times\RR$ like a Cauchy problem in free space
$(0,T)\times\R^{1+d}$.

From estimate \eqref{energy}, one obtains well-posedness of
Eq.~\eqref{ibvp2} in the basic Hilbert space $\cK^{0,\delta}(\RR)$,
i.e., the first part of Theorem~\ref{thm2} for $s=0$, $\sigma=0$, and
$\theta_0=0$. Well-posedness in the weighted Sobolev spaces
$\cK^{s,\delta}(\RR)$, i.e., the first part of Theorem~\ref{thm2} in
all other cases with $\theta_0=0$, then likewise follows using order
reductions that exist in the pseudodifferential calculus considered.

To establish the well-posedness results in the Sobolev spaces
$H_{P,\theta}^{s,\delta}(\oRR)$ with asymptotics is a considerably
more involved task. The crucial observation is that the boundary
traces $\gamma_{pk}u$ solve hyperbolic Cauchy problems in the lateral
boundary. To see this, one needs to know that, besides the compressed
principal symbol $\tilde\sigma_\psi^\mu(A)$, operators
$A\in\Psi_\cc^\mu(\oRR)$ possess also a sequence
$\left(\sigma_\cc^{-j}(A)\right)_{j\in\N_0}$ of so-called conormal
symbols. Like the function $h(z)=h(y,z,D_y)$ above, these are entire
functions of $z\in\C$ taking values in $\Psi_\cl^\mu(\R^d)$, and they
determine the manner in which asymptotics are mapped by~$A$. More
precisely, it holds that
\begin{equation}\label{gamma_au}
  \gamma_{pk}(Au) = \sum_{j\geq0}\sum_{\ell-r=k}
  \frac1{r!}\,\partial_z^r\sigma_\cc^{-j}(A)(p+j)\gamma_{p+j,\ell}(u),
\end{equation}
where the finite sum in the right-hand side is over those $(j,\ell,r)$
such that $\Re p+j<1/2-\delta$. Thus, applying $\gamma_{pk}$ to both
sides of the equation in \eqref{ibvp} results in an equation for
$\gamma_{pk}u$,
\[
\left\{ \enspace
\begin{aligned}
  & \partial_t(\gamma_{pk}u) + \sigma_\cc^0(\cA(t))(\gamma_{pk}u) =
  (\gamma_{pk}f)(t,y) + R_{pk}[u](t,y), \quad
  (t,y)\in(0,T)\times\partial\Omega, \\
  & (\gamma_{pk} u)\bigr|_{t=0} = (\gamma_{pk}u_0)(y),
\end{aligned}
\right.
\]
where the term $R_{pk}[u]$ has a similar meaning as in
\eqref{asymp_terms}. In fact, a compatibility condition between
$\sigma_\psi^1(\sigma_\cc^0(\cA(t)))$ and $\tilde
\sigma_\psi^1(\cA(t))$ ensures that the operator $\partial_t
+\sigma_\cc^0(\cA(t))$ is symmetrizable hyperbolic.

Hence, one obtains existence, uniqueness, and higher regularity for
the boundary traces $\gamma_{pk}$ in the correct regularity
classes. Subtracting the boundary terms from the prospective solution
$u$, one ends up in weighted Sobolev spaces
$\cK^{s-\theta_0,\delta+\theta_0}(\RR)$, in which well-posedness has
been shown before. Note that at this place it is crucial that the
cone-degenerate pseudodifferential operators in $\Psi_\cc^\mu(\oRR)$
have holomorphic conormal symbols (as opposed to finitely meromorphic
ones, one usually sees in a cone pseudodifferential calculus), which
implies that the action of $A\in \Psi_\cc^\mu(\oRR)$ on
$\cK^{s,\gamma}(\RR)$ is the same for any conormal order $\gamma\in\R$
(in the sense that it agrees on $\cK^{s,\gamma}(\RR)\cap
\cK^{s+\mu,\gamma'}(\RR)$ independently of whether this intersection
is seen as a subspace of $\cK^{s,\gamma}(\RR)$ or 
$\cK^{s,\gamma'}(\RR)$). Hence, the argument provided for
well-posedness in function spaces with conormal order $\delta$ works
for any other conormal order just the same.


\subsection{Comparison with other results and open problems}

One of the big open problems in the field is to provide satisfactory
answers concerning well-posedness for hyperbolic boundary problems
with a uniformly characteristic boundary. There only exist several
partial results in the literature, see
e.g. \cite{BGS2007,MO1975,CST2006,Rau1985,Sec1996}.

Totally characteristic hyperbolic boundary problems (for higher-order
scalar equations) were treat\-ed by
\textsc{Sakamoto}~\cite{Sak1989}. She obtained results comparable to
ours by showing (in our own notation) well-posedness in the scales
$H_{P,\theta}^{s,\delta}(\overline\Omega)$, where $P= T^\delta P_0$,
with $P_0$ being the type for Taylor asymptotics, and
$\theta=s\geq0$. Our results are slightly more general in that respect
that we now allow general asymptotic types $P$, general weight
intervals given by $\theta\geq0$, and also negative Sobolev
orders~$s$. In addition, we show that the boundary traces are given as
solutions to hyperbolic Cauchy problems in the lateral boundary. This
later result appears to be new.

\textsc{Sakamoto}~\cite{Sak1989} used pseudodifferential techniques to
establish her results, albeit in a different manner. Our approach
might have the advantage that it yields a symmetrizer also in the
uniformly characteristic case, upon further developing an adapted
pseudodifferential calculus. It seems to be evident that this calculus
has to be some sort of an edge calculus, as in
\textsc{Schulze}~\cite{Sch1998}. On the level handled in this paper
the difference between an edge calculus and a cone calculus is rather
marginal: \ In the latter one first performs the Fourier transform
$\mathcal F_{y\to\eta}$ with respect to the $y$-variables and then the
Mellin transform $M_{x\to z}$ with respect to the $x$-variable, while
in the former these transforms are performed in the opposite order. As
both operations commute, $M_{x\to z} \mathcal F_{y\to\eta} = \mathcal
F_{y\to\eta} M_{x\to z}$, it is possible to recast the cone calculus
utilized here in the form of an edge calculus, up to some technical
details.
Note that the form of the asymptotic terms as given in \eqref{asp4} is
already typical of edge problems.



\subsection{Notation}\label{notat}

We shall use freely standard notation from microlocal analysis (see
\cite{Hoe2007}). For singular analysis, we closely follow the notation
used in \cite{Sch1998}.

Throughout the paper, we shall especially employ the following
notation:

\begin{itemize}
\item $\delta\in\R$ denotes the \textit{reference conormal order\/},
  which is fixed once and for all.
  
\item For the Mellin covariable $z\in \C$, we write $z=\beta+\ii \tau$
  with $\beta,\,\tau \in \R$.

\item For $\beta\in\R$, we introduce the weight line
  $\Gamma_\beta=\{z\in\C\mid\Re z=\beta\}$. The weight line that
  corresponds to the reference conormal order $\delta$ is
  $\Gamma_{1/2-\delta}$.

\item $\varphi\in\cC^\infty(\overline\R_+;\R)$ is a cutoff function,
  i.e., $0\leq \varphi\leq 1$, $\varphi(x)= 1$ for $|x|\lesssim1$, and
  $\varphi(x)=0$ for $|x|\gtrsim1$. Likewise, $\varphi_0,\varphi_1$
  are also cutoff functions satisfying, in addition,
  $\varphi\varphi_0=\varphi$ and $\varphi\varphi_1=\varphi_1$ (i.e.,
  $(1-\varphi)(1-\varphi_1)=1-\varphi$).

\item $\psi\in\cC^\infty(\overline{\R}_+;\R)$ denotes a non-decreasing
  function such that $\psi(x) = x$ for $0\leq x\leq 1/2$ and
  $\psi(x)=1$ for $x\geq1$. Furthermore, $\psi^\varrho$ for
  $\varrho\in\R$ is the $\varrho$th power of $\psi$.

\item $\langle \eta\rangle= (4+ |\eta|^2)^{1/2}$ for
  $\eta\in\R^d$. Hence, $\langle \eta\rangle\geq2$ and $\log \langle
  \eta\rangle>0$.

\item The Mellin transform of $u$ with respect to $x\in\Rp$ is
  $\widetilde u(z) = Mu(z) = \int_0^\infty x^{z-1} u(x)\,\dd x$ for
  $z\in\C$. The inverse Mellin transform is $M^{-1}v(x) =
  \frac1{2\pi\ii} \int_{\Gamma_\beta} x^{-z}v(z)\,\dd z$ for a
  suitable $\beta\in\R$ depending on the situation under
  consideration.
  
\item The Fourier transform of $w$ with respect to $y\in\R^d$ is $\hat
  w(\eta) = \mathcal F w(\eta) = \int_{\R^d}\ee^{-\ii
    y\cdot\eta}w(y)\,\dd y$ for $\eta\in\R^d$. The inverse Fourier
  transform is $\mathcal F^{-1}\omega(y) = \int_{\R^d} \ee^{\ii y\cdot
    \eta} \omega(\eta)\,\dbar \eta$, where $\dbar \eta
  =(2\pi)^{-d}\,\dd\eta$.

\item The space $H^{s,\langle k\rangle}(\R^d)$ for $(s,k)\in\R\times\Z$,
  consists of all $w$ such that $\langle\eta\rangle^s \log^k
  \langle\eta\rangle \hat w(\eta)\in L^2(\R^d)$. In particular,
  $H^s(\R^d)= H^{s,\langle 0\rangle}(\R^d)$.

\item In the closed half-space $\oRR$, we use coordinates $(x,y)$
  with $x\geq0$ and $y\in\R^d$.

\item The weighted Sobolev spaces $\cK^{s,\gamma}(\RR)$ are introduced
  in Definition~\ref{Def3.6} and the Sobolev spaces
  $H_{P,\theta}^{s,\delta}(\oRR)$ with asymptotics in
  Definition~\ref{Def3.10}.

\item $\gamma_{pk}$ for $(p,k)\in P$, where $P$ is an asymptotic type,
  is a trace operator. Similarly, $\Gamma_{pk}$ for $(p,k)\in P$ is a
  potential operator.

\item $\mathcal M^\mu(\R^d)$ denotes the space of holomorphic Mellin
  symbols $h(z) = h(y,z,D_y)$.
  
\item The Mellin quantization of an amplitude function $h\in
  \cC^\infty(\overline\R_+;\mathcal M^\mu(\R^d))$ is
  \[
    (\operatorname{op}_M(h)u)(x,y) =
  \frac1{2\pi\ii}\int_{\Gamma_{1/2-\delta}} \, x^{-z}\,
  h(x,y,z,D_y)\tilde u(z,y)\,\dd z.
  \]
  $\operatorname{op}_M(h)$ acts on suitable distributions $u=u(x,y)$.
  
\item The Fourier quantization of an amplitude function $a\in
  S_\cl^\mu(\R^d\times\R^d)$ is $a(y,D_y)=\operatorname{op}_\psi(a)$
  with
  \[
  (\operatorname{op}_\psi(a)v)(y) = \int_{\R^d} e^{i y
    \cdot \eta} \, a(y,\eta) \, \hat{v}(\eta) \,\dbar\eta,
  \]
  where $\dbar\eta =
  (2\pi)^{-d}\,\dd \eta$. The principal symbol of $A=A(y,D_y)$ is
  denoted by $\sigma_\psi^\mu(A)$.

\item $\Psi_\cc^\mu(\oRR)$ is the class of cone-degenerate
  pseudodifferential operators, of order $\mu\in\R$, utilized
  here. Elements of this space are symbolically written as $\mathcal
  A(x,y,xD_x,D_y)$.

\item $\widetilde T^*\oRR$ denotes the compressed cotangent bundle over $\oRR$.

\end{itemize}


\section{Well-posedness in standard Sobolev spaces}\label{sec2}

We start with proving Theorem~\ref{thm1}. Notice that it is enough to
treat the case $\sigma=0$. Cases with $\sigma\geq1$ then follow by
differentiating the equation $\sigma$ times with respect to $t$, as in
the proof of Proposition~\ref{abcd} below.

\begin{prop}
Let $u_0\in H^s(\Omega;\C^N)$, $f\in L^1((0,T);H^s(\Omega;\C^N))$ for
some $s\geq0$. Then \textup{Eq.~\eqref{ibvp}} possesses a unique
solution $u\in \cC([0,T]; H^s(\Omega;\C^N))$. In addition, for
$\ell<s-1/2$, one has that $\gamma_\ell u\in
\cC([0,T];H^{s-\ell-1/2}(\partial\Omega;\C^N))$ is uniquely determined
as the solution to the hyperbolic Cauchy problem
\begin{equation} \label{trace1}
\left\{ \enspace  
\begin{aligned}
  & \p_t (\gamma_\ell u) + \sum_{j=1}^d
  A_j(t,0,y)\partial_j(\gamma_\ell u) + \left(\ell\, A(t,0,y) +
  B(t,0,y) \right)\gamma_\ell u \\
  & \hspace*{090mm} = (\gamma_\ell f)(t,y)+ R_\ell[u](t,y),
  \\ & (\gamma_\ell u)\bigr|_{t=0} = (\gamma_\ell u_0)(y).
\end{aligned}
\right.
\end{equation}
Here, the term $R_\ell[u]$ is zero for $\ell=0$ and linear in
$\gamma_0u,\dotsc,\gamma_{\ell-1}u$ for $\ell\geq1$.
\end{prop}

The precise form of the term $R_\ell[u]$ will be given in \eqref{taking_traces}
below.

\begin{proof}
Extend the matrix-valued coefficients $A, A_j, B$ in Eq.~\eqref{ibvp}
to matrix-valued functions \linebreak $A, A_j, B\in \cC^\infty([0,T]\times
\R^{n};M_{N\times N}(\C))$ so as to obtain a uniformly
symmetrizable hyperbolic system $\partial_t+xA\partial_x +
\sum_{j=1}^d A_j\partial_j +B$ on $(0,T)\times \R^{n}$ (keeping the
notation from above). Then consider the hyperbolic Cauchy problem
\begin{equation}\label{ext}
\left\{  
\begin{aligned}
  \partial_t U + x A(t,x,y)\partial_x U+ \sum_{j=1}^d
  A_j(t,x,y)\partial_jU + B(t,x,y)U & = F(t,x,y), \quad (t,x,y)\in
  (0,T)\times \R^{n}, \\
  U\bigr|_{t=0} & = U_0(x,y),
\end{aligned}
\right.
\end{equation}
where $U_0\in H^s(\R^{n};\C^N)$, $U_0\bigr|_{\Omega}=u_0$, $F\in
L^1((0,T);H^s(\R^{n};\C^N))$, and
$F\bigr|_{(0,T)\times\Omega}=f$. Because this system is on whole
space, Eq.~\eqref{ext} possesses a unique solution $U\in \cC([0,T];
H^s(\R^{n};\C^N))$. As characteristics of this system are tangent to
the hypersurface $(0,T)\times\partial\Omega$, it follows that whatever
starts out in the region in which $(0,T)\times(\R^n\setminus\Omega)$
stays in that region for all times.  Therefore,
$u=U\bigr|_{(0,T)\times\Omega}$ only depends on $u_0$, $f$, in
particular, $u$ is independent of all the extensions chosen. We
conclude that $u$ is the unique solution to the original
problem~\eqref{ibvp2}.

Differentiating \eqref{ext} a number of times with respect to $x$ and
setting $x=0$ yields \eqref{trace1}.
\end{proof}


\section{Cone-degenerate pseudodifferential operators}\label{CDPSDO}

The main technical tool to prove the result in Theorem~\ref{thm2} is a
calculus for a certain class of cone-degenerate pseudodifferential
operators on $\RR$. Here we briefly introduce this pseudodifferential
calculus. Calculi for cone-degenerate pseudodifferential operators
have been developed by \textsc{B.-W.~Schulze} \cite{Sch1991,Sch1998},
see also \cite{HS2008}. We closely follow his approach and refer to
the said references for details. An equivalent calculus is the
$b$-calculus of \textsc{R.~Melrose and G.~Mendoza}
\cite{Mel1993,MM1983}. For our purposes, Schulze's cone calculus is
preferable as it is more analytic in flavor and, therefore, easier to
adapt to our needs.

Compared to \cite{Sch1991,Sch1998}, there are a few differences. First
of all, the base of the cone is $\R^d$ which is a non-compact
manifold. This non-compactness introduces no additional difficulties,
as we are not interested in the Fredholm property of elliptic
operators, but solely in the construction of a symmetrizer. Secondly,
the coefficients $u_{pk}$ in the asymptotic expansions \eqref{asp1} do
not belong to finite-dimensional subspaces of $\cC^\infty(\R^d;\C^N)$,
but instead can be any functions from the space $H^{s+\Re
  p+\delta-1/2,\langle k\rangle}(\R^d;\C^N)$, see
Definition~\ref{Def3.10} and Proposition~\ref{trace_thm} for
details. This then requires a special treatment of the asymptotic
terms, see Definition~\ref{Def3.9}. In fact, the function spaces
$H_{p,\theta}^{s,\delta}(\oRR;\C^N)$ employed below are modeled after
the edge Sobolev spaces of Schulze (see \cite{LRW2022} for a
discussion of this point). Lastly, the cone-degenerate
pseudodifferential operators we utilize do not produce any further
asymptotic information, but instead preserve the given one. This is in
the sense that the given asymptotic type, $P$, which collects the
$(p,k)$ appearing in \eqref{asp1}, is preserved, while certainly the
coefficients $u_{pk}$ are, in general, altered when applying an
operator~$A$ belonging to the calculus to $u$ (see
Proposition~\ref{ae3} for the way in which this happens). Accordingly,
the conormal symbols $\sigma_\cc^{-j}(A)$ of operators $A$ in the
calculus (see Definition~\ref{Def3.16}) are holomorphic functions of
the Mellin covariable~$z\in\C$, while for general cone
pseudodifferential calculi these conormal symbols are finitely
meromorphic functions of~$z\in\C$. Again, our choice is justified by
the fact that we do not have to construct parametrices for elliptic
cone-degenerate pseudodifferential operators (with the exception of
Proposition~\ref{Lem4.3} where we establish the existence of order
reductions).


\subsection{Asymptotic types}

The functional-analytic approach of handling the asymptotic expansions
\eqref{asp2} starts with collecting the data $(p,k)$ appearing in
\eqref{asp2} into so-called asymptotic types. Recall that we fix a
$\delta\in\R$ as a reference conormal order.

\begin{defn}  \label{Def3.1}
The set $\underline{\operatorname{As}}^\delta$ of \textit{asymptotic
  types\/} associated with the conormal order $\delta\in\R$ consists
of discrete subsets $P\subset \C\times\N_0$ with the following
properties:
\begin{enumerate} [(i)]
\item $\Re p<1/2-\delta$ for $(p,k)\in P$,
\item $\Re p\to -\infty$ as $(p,k)\in P$, $|p|\to\infty$,
\item $(p,k-1)\in P$ if $(p,k)\in P$ and $k>0$,
\item $(p-1,k)\in P$ if $(p,k)\in P$.
\end{enumerate}
\end{defn}

\begin{rem}
Property (iv) is needed to guarantee the coordinate invariance of the
constructions.
\end{rem}

We set $\pi_\C P =\bigl\{p\in\C \mid \text{$(p,k)\in P$ for some $k\in
  \N_0$}\bigr\}$ and $m_p=\max\{k+1\mid(p,k)\in P\}$ with the
convention that $m_p=0$ if $p\notin \pi_\C P$. An asymptotic type $P$
is then completely determined by the $m_p$. In this sense, $P$ can be
thought of as a non-negative divisor (in the sense of complex
analysis) having additional properties. Still, regarding $P$ as a
discrete subset of $\C\times\N_0$ comes in handy in the notation
employed below.
 
\begin{ex}
\begin{enumerate}[(i)]
\item The asymptotic type governing \eqref{taylor} is
  $P_0=\{(-\ell,0)\mid \ell\in\N_0\}$. We refer to it as
  \textit{Taylor asymptotics\/}.
\item As a subset of $\C\times\N_0$, the \textit{empty asymptotic
  type\/} $\mathcal O$ is given by $\mathcal O=\emptyset$.
\end{enumerate}
\end{ex}

Let $P\in \underline{\operatorname{As}}^\delta$, $\varrho\in\R$. Then
we define $T^\varrho P\in
\underline{\operatorname{As}}^{\delta+\varrho}$ to consists of all
$(p,k)\in \C\times\N_0$ such that $(p+\varrho,k)\in P$.


\subsection{Function spaces}
 
Next we introduce suitable weighted Sobolev spaces and Sobolev spaces
with asymptotics. These are the function spaces in which we will
establish the energy inequalities.


\subsubsection{Weighted Sobolev spaces}

\begin{defn} \label{Def3.3}
Let $\gamma\in\R$. For $s\in\N_0$, the weighted Sobolev space
$\cH^{s,\gamma}(\R_+^{1+d})$ is defined to consist of all functions
$u=u(x,y)$ such that
\[
   x^{-\gamma} (x \partial_x)^j \partial_y^\alpha u \in
   L^2(\R_+^{1+d}), \quad j+|\alpha|\leq s.
\]
For general $s\in\R$, the spaces $\cH^{s,\gamma}(\R_+^{1+d})$ are then
introduced by complex interpolation and duality.
\end{defn} 

We can characterize the space $\cH^{s,\gamma}(\R_+^{1+d})$ via the
Mellin transform (see Section~\ref{a1}).

\begin{lem}  \label{Lem3.4}
Let $s\geq0$, $\gamma\in\R$. Then $u\in \cH^{s,\gamma}(\R_+^{1+d})$ if
and only if
\[
   \frac1{2\pi\ii}\int_{\Gamma_{1/2-\gamma}} \left(\|\tilde
   u(z,\cdot)\|_{H^s(\R^d)}^2 + \langle z\rangle^{2s}\|\tilde
   u(z,\cdot)\|_{L^2(\R^d)}^2\right)\,\dd z <\infty,
\]
where $\tilde u(z, \cdot)$ is the Mellin transform of $u(x,
\cdot)$ with respect to $x$.
%
\end{lem}
\medskip

As said in the introduction, we are mostly interested in the behavior
of solutions to Eq.~\eqref{ibvp2} near $x=0$. Hence, we make a generic
choice for their possible behavior near $x=\infty$ (also compare with
Eq.~\eqref{ppt3}). Note that $u \in \cH^{s,\gamma}(\R_+^{1+d})$
implies that $\varphi u \in \cH^{s,\gamma}(\R_+^{1+d})$.

\begin{defn} \label{Def3.6}
For $s,\gamma\in\R$, we set 
\[
  \mathcal K^{s,\gamma}(\R_+^{1+d}) = \bigl\{u\bigm| \varphi u\in
  \cH^{s,\gamma}(\R_+^{1+d}), \;
  \left(1-\varphi\right)u\in H^s(\R_+^{1+d})\bigr\}.
\]
\end{defn}

In view of $\cH^{s,\gamma}(\R_+^{1+d}) \subset
H^s_{\text{loc}}(\R_+^{1+d})$, the space $\mathcal
K^{s,\gamma}(\R_+^{1+d})$ is independent of the choice of the cut-off
function $\varphi$. Moreover, $\mathcal K^{s,\gamma}(\R_+^{1+d})$ is a
Hilbert space in a natural way.

\smallskip

Now fix $\delta\in\R$. In the sequel,
\[
  \text{$\mathcal K^{0,\delta}(\R_+^{1+d})$ will serve as reference
    Hilbert space.}
\]
Write $\langle \;,  \,\rangle$ for the inner product
and  $\|\;\|$ for the norm in  $\mathcal K^{0,\delta}(\R_+^{1+d})$.


\subsubsection{The asymptotic terms}

The next two lemmas prepare for introducing the terms occurring in the
asymptotic expansions \eqref{asp2}. Recall that the asymptotic
expansions \eqref{asp2} are formal in the sense that, in general, we
do not have enough regularity for the coefficients~$v_{pk}$ to write
the asymptotic terms as tensor products.
\begin{lem}\label{Lem3.5}
Let $(p,k)\in \C\times\N_0$ and $w\in H^s(\R^d)$ for some
$s\in\R$. Then the function $v$ defined by
\[
  v(x,y) = \mathcal F^{-1}_{\eta\to y}
  \bigl\{\varphi(x\ler)(x\ler)^{-p}
  \log^k(x\ler)\hat w(\eta)\bigr\}
\]
belongs to $\bigcap_{\epsilon>0} \cK^{s+\Re p+\epsilon,1/2-\Re
  p-\epsilon}(\RR)$. Moreover,
\begin{equation}\label{nunu}
  v(x,y) = \varphi(x) \bigl[\left( x\langle D_y\rangle\right)^{-p}
  \log^k\!\left(x\langle D_y\rangle\right) w\bigr](y) + v'(x,y),
\end{equation}
where $v'\in \bigcap_{\epsilon>0} \cK^{s+\Re p-\epsilon,1/2-\Re
  p+\epsilon}(\RR)$.
\end{lem}
\begin{proof}
Let $m(z)$ denote the Mellin transform of
$\varphi(x)x^{-p}\log^k\!x$. Then $m(z)$ is meromorphic in $\C$ with a
single pole of order $k+1$ at $z=p$. In addition, $m(z) = (-1)^k k!
\left(z-p\right)^{-(k+1)} + O(1)$ as $z\to p$ and $(\chi m)(z)\in
\cC^\infty(\R_\beta;\mathcal S(\R_ \tau))$, where $\chi\in
\cC^\infty(\C)$, $\chi(z)=0$ for $|z-p|\leq 1/2$, and $\chi(z)=1$ for
$|z-p|\geq 1$ (see Lemma~\ref{ppl}).  Recall that we have written
$z=\beta+\ii \tau$ with $\beta,\,\tau \in \R$.

Direct calculations show that
\[
  \tilde v(z,y) = \mathcal F^{-1}_{\eta\to
    y}\bigl\{\ler^{-z}m(z)\hat w(\eta)\bigr\}.
\]
In particular, $\tilde v(z,\cdot)$ is meromorphic in $z\in\C$ taking
values in $H^{-\infty}(\R^d)$ with a single pole of order $k+1$ at
$z=p$ and
\begin{equation}\label{nun}
  \tilde v(z,\cdot) = \sum_{r=0}^k\frac{(-1)^{k-r}
    k!}{r!}\,(z-p)^{-(k-r+1)}\,\langle D\rangle^{-p} \log^r\langle
  D\rangle w+ O(1) \quad \text{as $z\to p$}
\end{equation}
in view of $\ler^{-z} = \sum_{r=0}^k
\frac{(-1)^r}{r!}\,\ler^{-p}\log^r\ler\,(z-p)^r + O((z-p)^{k+1})$ as
$z\to p$.

\smallskip

For all $t\in\R$ and $\epsilon>0$, 
\begin{multline*}
  \frac1{2\pi\ii}\int_{\Gamma_{\Re p+\epsilon}} \langle z\rangle^{2t}
  \|\tilde v(z,\cdot)\|_{H^{s+\Re p+\epsilon}(\R^d)}^2\,\dd z  \\ = \, \frac1{2\pi\ii}
  \int_{\Gamma_{\Re p+\epsilon}} \langle z\rangle^{2t}
  |m(z)|^2  \left( \int_{\R^d}
  \langle\eta\rangle^{2(s+\Re p+\epsilon)} \langle\eta\rangle^{-2\Re z}
  |\hat w(\eta)|^2\,\dbar \eta\right)\dd z <\infty.
\end{multline*}
This implies that $v\in \cH^{s+\Re p+\epsilon,1/2-\Re
  p-\epsilon}(\R_+^{1+d})$ for $s+\Re p+\epsilon\geq0$ by
Lemma~\ref{Lem3.4} and then for $s+\Re p+\epsilon<0$ by duality.

\smallskip

Denote $v''(x,y) = (2\pi\ii)^{-1}\int_{\Gamma_{\Re p-\epsilon'}}x^{-z}
\tilde v(z,y)\,\dd z$, where $\epsilon'>0$ is arbitrary. Repeating the
argument just given, one finds that $v''\in \bigcap_{\epsilon>0}
\cH^{s+\Re p-\epsilon,1/2-\Re p+\epsilon}(\RR)$. Furthermore, by
\eqref{nun} and Cauchy's integral theorem,
\[
  v(x,\cdot) - v''(x,\cdot) = \bigl[\left( x\langle D_y\rangle\right)^{-p}
  \log^k\!\left(x\langle D_y\rangle\right) w\bigr](y).
\]
One obtains \eqref{nunu} by multiplying the last equation by
$\varphi(x)$ and taking into account that $\varphi(x)\varphi(x\ler)=
\varphi(x\ler)$ for all $\eta\in\R^d$.
\end{proof}

\begin{lem} \label{Lem3.8}
Let $(p,k)\in\C\times\N_0$. Suppose that $w\in H^{s,\langle
  k\rangle}(\R^d)$ for some $s\in\R$. Then
\[
  \mathcal F^{-1}_{\eta\to y}\{\varphi(x\ler)\hat{w}(\eta)\}x^{-p}\log^k\!x \in
   \bigcap_{\epsilon>0}\cK^{s+\epsilon,1/2-\Re p-\epsilon}(\R_+^{1+d})
\] 
and
\[
  \mathcal F^{-1}_{\eta\to
    y}\{\varphi(x\ler)\hat{w}(\eta)\}x^{-p}\log^k\!x -
  \varphi(x)x^{-p}\log^k\! x\,w(y)\in
  \bigcap_{\epsilon>0}\cK^{s-\epsilon,1/2-\Re p+\epsilon}(\R_+^{1+d}).
\]
\end{lem}

\begin{proof}
We set $w_0 = \mathcal F^{-1} \{\ler^p\hat{w}(\eta)\}\in H^{s-\Re p, \langle
  k\rangle}(\R^d)$ and proceed by induction on $k$.

For $k=0$, one has
\[
  \mathcal F^{-1}_{\eta\to y}\{\varphi(x\ler)\hat{w}(\eta)\}x^{-p} =
  \mathcal F^{-1}_{\eta\to y} \bigl\{\varphi(x\ler) (x\ler)^{-p}
  \hat{w}_0(\eta)\bigr\}\in
  \bigcap_{\epsilon>0}\cK^{s+\epsilon,1/2-\Re p-\epsilon}(\R_+^{1+d})
\]
and
\[
  \mathcal F^{-1}_{\eta\to y}\{\varphi(x\ler)\hat{w}(\eta)\}x^{-p} -
  \varphi(x)x^{-p}w(y) \in \bigcap_{\epsilon>0}\cK^{s-\epsilon,1/2-\Re
    p+\epsilon}(\R_+^{1+d})
\]
by Lemma~\ref{Lem3.5}.
 
For $k\geq1$, we set $w_r = \mathcal F^{-1}_{\eta\to y} \{ \log^r\ler \,
\hat{w}(\eta)\}\in H^{s,\langle k-r\rangle}(\R^d)$ for $1\leq r\leq k$.
Then, by Lemma~\ref{Lem3.5} and the induction hypothesis,
\begin{multline*}
  \mathcal F_{\eta\to
    y}^{-1}\{\varphi(x\ler)\hat{w}(\eta)\}x^{-p}\log^k\!x = \mathcal
  F_{\eta\to y}^{-1} \bigl\{\varphi(x\ler)(x\ler)^{-p}\log^k(x\ler)
  \hat{w}_0(\eta)\bigr\} \\ - \sum_{r=1}^k \binom{k}{r}\mathcal F_{\eta\to y}^{-1}
  \bigl\{\varphi(x\ler) \hat{w}_r(\eta)\bigr\} x^{-p}\log^{k-r}\!x \in
  \bigcap_{\epsilon>0}\cK^{s+\epsilon,1/2-\Re p-\epsilon}(\R_+^{1+d})
\end{multline*}
as well as 
\begin{multline*}
  \mathcal F_{\eta\to
    y}^{-1}\{\varphi(x\ler)\hat{w}(\eta)\}x^{-p}\log^k\!x -
  \varphi(x)x^{-p}\log^k\! x\,w(y) \\ = \left( \mathcal F_{\eta\to
    y}^{-1} \bigl\{\varphi(x\ler)(x\ler)^{-p}\log^k(x\ler)
  \hat{w}_0(\eta)\bigr\} - \varphi(x) \bigl[\left( x\langle
    D_y\rangle\right)^{-p} \log^k\!\left(x\langle D_y\rangle\right)
    w_0\bigr](y) \right) \\ - \sum_{r=1}^k \binom{k}{r}\left( \mathcal
  F_{\eta\to y}^{-1} \bigl\{\varphi(x\ler) \hat{w}_r(\eta)\bigr\}
  x^{-p}\log^{k-r}\!x -\varphi(x)x^{-p}\log^{k-r}\!x\,w_r(y)\right)
  \\ \in \bigcap_{\epsilon>0}\cK^{s-\epsilon,1/2-\Re
    p+\epsilon}(\R_+^{1+d}).
\end{multline*}
This finishes the proof.
\end{proof}

We are now ready to introduce potential operators. 
\begin{defn} \label{Def3.9}
For $(p,k)\in \C\times\N_0$, the \textit{potential operator\/}
$\Gamma_{pk}$ acting on functions $w=w(y)$ is given by
\begin{equation}\label{nmu}
   (\Gamma_{pk} w)(x,y) = \frac{(-1)^k}{k!}\,\mathcal
    F^{-1}\left\{ \varphi(x\langle\eta\rangle)\hat{w}(\eta)\right\}
    x^{-p}\log^k\!x.
\end{equation}
\end{defn}

The role played by the normalizing factor $(-1)^k/k!$ becomes apparent
from Lemma~\ref{paul} in conjunction with \eqref{impo}. Based on
Lemma~\ref{Lem3.8}, we have that, for any $s\in\R$,
\[
  \Gamma_{pk}\colon H^{s,\langle k\rangle}(\R^d) \to
  \bigcap_{\epsilon>0} \cK^{s+\epsilon,1/2-\Re p-\epsilon}(\R_+^{1+d}).
\]


\subsubsection{Sobolev spaces with asymptotics}\label{SSWA}

The definition of the Sobolev spaces with asymptotics postu\-lates an
improvement of the conormal order up to $\delta+\theta$ upon
subtracting finitely many asymptotic terms. This works as long as no
singular exponent $p$ from the asymptotic type $P$ comes to lay on the
weight line $\Gamma_{1/2-\delta-\theta}$. This leaves out a discrete
set of values for $\theta>0$. The general case is then handled by
complex interpolation.

\begin{defn} \label{Def3.10}
Let $s\in\R$, $P\in \underline{\textup{As}}^\delta$, and
$\theta\geq0$.

(i) \ For $\pi_\C P\cap \Gamma_{1/2-\delta-\theta}=\emptyset$, the
space $H_{P,\theta}^{s,\delta}(\overline\R_+^{1+d})$ consists of all
$v\in\mathcal K^{s,\delta}(\R_+^{1+d})$ for which there are functions
$v_{pk} \in H^{s+\Re p+\delta-1/2,\langle k\rangle}(\R^d)$ for
$(p,k)\in P$, $\Re p>1/2-\delta-\theta$ such that
\begin{equation}\label{decomp}
 v- \sum_{\substack{(p,k)\in P,\\\Re p>1/2-\delta-\theta}}
 \Gamma_{pk} v_{pk} \in
 \cK^{s-\theta,\delta+\theta}(\R_+^{1+d}).
\end{equation}

(ii) \ For general $\theta\geq0$, the space
$H_{P,\theta}^{s,\delta}(\overline\R_+^{1+d})$ is then defined by
complex interpolation with respect to the parameter $\theta$.
\end{defn}

For $s\geq0$, we also write
$H_P^{s,\delta}(\overline\R_+^{1+d})=H_{P,s}^{s,\delta}(\overline\R_+^{1+d})$.

\begin{ex} \label{rem3.11}
The spaces $H_{P,\theta}^{s,\delta}(\overline\R_+^{1+d})$ constitute a
natural generalization of the standard Sobolev spaces in view of the
following two facts: 

(i) \ For $s \geq0$, $H^s(\R_+^{1+d}) =
  H_{P_0}^{s,0}(\overline{\R}_+^{1+d})$, where $P_0$ is the type for
  Taylor asymptotics.

(ii) \ For $s\geq0$,
  $H_0^s(\overline{\R}_+^{1+d}) = H_{\mathcal O}^{s,0}(\overline{\R}_+^{1+d})$,
  where $\mathcal O$ is the empty asymptotic type. 
\end{ex}

It is not hard to see that the coefficients $v_{pk}$ in \eqref{decomp}
are uniquely determined. We then introduce, for $(p,k)\in P$ with $\Re
p>1/2-\delta-\theta$, the \textit{trace operators\/}
\begin{equation}\label{trace_op}
  \gamma_{pk} \colon H_{P,\theta}^{s,\delta}(\overline\R_+^{1+d}) \to
  H^{s+\Re p+\delta-1/2,\langle k\rangle}(\R^d), \quad v \mapsto v_{pk}.
\end{equation}
Essentially by definition, we have the following trace theorem.

\begin{prop}\label{trace_thm}
Let $s\in\R$, $P\in \underline{\textup{As}}^\delta$, $\theta\geq0$,
and $\pi_\C P\cap \Gamma_{1/2-\delta-\theta}=\emptyset$. Then the
short sequence
\[
  0 \longrightarrow H_{\mathcal O,\theta}^{s,\delta}(\oRR)
  \longrightarrow H_{P,\theta}^{s,\delta}(\oRR)
  \xrightarrow{(\gamma_{pk})} \bigoplus_{\substack{(p,k)\in
      P,\\\Re p>1/2-\delta-\theta}}H^{s+\Re
    p+\delta-1/2,\langle k\rangle}(\R^d) \longrightarrow 0
\]
is split exact.
\end{prop}

The next result tells us that, for theoretical purposes, one can
adjust the reference conormal order $\delta$ to be any given real number.

\begin{lem}\label{lem_310}
Let $s\in\R$, $P\in \underline{\textup{As}}^\delta$, $\theta\geq0$,
and $\varrho\in\R$. Then multiplication by $\psi^\varrho(x)$ realizes
an isomorphism between $H_{P,\theta}^{s,\delta}(\overline\R_+^{1+d})$
and $H_{T^\varrho P,\theta}^{s,\delta+\varrho}(\overline\R_+^{1+d})$.
\end{lem}
\begin{proof}
A straightforward verification.
\end{proof}  


\subsubsection{The Schwartz class with asymptotics of type $P$}

At last, we introduce a replacement for the space $\cS(\oRR)=
\bigl\{v |_{\,\RR}\bigm| v\in \cS(\R^{1+d})\bigr\}$. First we let
$\cS_0(\oRR)$ be the space of all $u\in \cS(\oRR)$ that vanish to
infinite order at $x=0$.

\begin{defn}
For $P\in\underline{\textup{As}}^\delta$, the space $\cS_P(\oRR)$
consists of all $v$ for which there are sequences
$(v_{pk})_{(p,k)\in P} \subset \cS(\R^d)$ and $(c_p)_{p\in\pi_\C
  P}\subset \R_+$ with $c_p\to \infty$ as $\Re p\to-\infty$
sufficiently fast such that
\begin{equation}\label{defnsp}
   v(x,y) - \sum_{(p,k)\in P} \frac{(-1)^k}{k!}\,\varphi(c_p x) x^{-p}
   \log^k\! x\, v_{pk}(y) \in \cS_0(\oRR).
\end{equation}
\end{defn}

The space $\cS_P(\oRR)$ is a nuclear Fr\'echet space in a natural
way. Furthermore, upon an appropriate choice of $(c_p)$ depending on
$(v_{pk})$, the series in the left-hand side of \eqref{defnsp}
converges absolutely in $\cS_P(\oRR)$. Moreover, $\cS_P(\oRR)\subseteq
H_{P,\theta}^{s,\delta}(\oRR)$ for any $s,\theta$ and, for $v$
as in \eqref{defnsp} and $(p,k)\in P$, we have
$\gamma_{pk}v=v_{pk}$. Indeed, the short sequence
\[
   0 \longrightarrow \cS_0(\oRR) \longrightarrow \cS_P(\oRR)
   \xrightarrow{(\gamma_{pk})_{(p,k)\in P}} \bigoplus_{(p,k)\in P}
   \cS(\R^d) \longrightarrow 0
\]
is exact.

\begin{ex}
One has $\cS_{P_0}(\oRR)=\cS(\oRR)$ and $\cS_{\mathcal O}(\oRR)=\cS_0(\oRR)$.
\end{ex}

\begin{lem}\label{dense}
The space $\cS_P(\oRR)$ is dense in $H_{P,\theta}^{s,\delta}(\oRR)$.
\end{lem}
\begin{proof}
By complex interpolation, we can assume that $\pi_\C P\cap
\Gamma_{1/2-\delta-\theta}=\emptyset$. It is known that
$\cS_0(\oRR)$ is dense in $\cK^{s-\theta,\delta+\theta}(\RR)$. In view
of \eqref{decomp} and as $\cS(\R^d)$ is dense in $H^{r,\langle
  l\rangle}(\R^d)$ for any $(r,l)\in\R\times\N_0$, it is enough to
show that $\Gamma_{pk}w\in \cS_P(\oRR)$ for $(p,k)\in P$ and $w\in
\cS(\R^d)$. The latter, in turn, will follow from the relation
\[
  \mathcal F_{\eta\to y}^{-1}\left\{\varphi(x\ler)\hat{w}(\eta)\right\}x^{-p}\log^k\!x -
  \varphi(x)x^{-p}\log^k\!x\, w(y)\in \cS_0(\oRR),
\]
which, however, is apparently true.
\end{proof}


\subsection{Calculus of cone-degenerate pseudodifferential operators}

We now introduce the class \linebreak $\Psi_\cc^\mu(\oRR)$ of
cone-degenerate pseudodifferential operators on the half-space
$\R_+^{1+d}$ mentioned in the introduction. To make connection to the
theory of cone-degenerate pseudodifferential operators, note that the
closed half-space $\oRR$ is considered as a blowup of the cone
$(\overline{\R}_+\times\R^d)/(\{0\}\times\R^d)$. We do not provide
proofs for results that can be found in the literature in the form as
stated or in a similar form (then with no essential changes in the
proofs). Notable exceptions are Propositions~\ref{prop3.20}
through \ref{Prop3.21}. For other
results, we refer to the literature, e.g., \textsc{Harutyunyan and
  Schulze} \cite{HS2008} or \textsc{Schulze} \cite{Sch1998}.


\subsubsection{Parameter-dependent pseudodifferential operators}\label{parameter-dep}

We start with parameter-dependent pseudodifferential operators.

\begin{defn}\label{Def3.12}
For $\mu\in\R$, the class $\Psi_\cl^\mu(\R^d;\R)$ of classical
parameter-dependent pseudodifferential operators on $\R^d$, with
parameter $\tau\in\R$, consists of all families
$A=\left(A(\tau)\right)_{\tau\in\R}\subset \Psi_\cl^\mu(\R^d)$ such
that
\[
  A(\tau) u(y) = \int_{\R^d} \ee^{\ii y \cdot \eta} a(y,\eta,\tau)\hat
  u(\eta)\,\dbar \eta, \quad y\in\R^d,
\]
where $a\in S_\cl^\mu(\R_y^d\times\R_{(\eta,\tau)}^{d+1})$.
\end{defn}

Note that any $A\in \Psi_\cl^\mu(\R^d;\R)$ admits a
parameter-dependent principal symbol $\sigma_\psi^\mu(A) \in
S^{(\mu)}(\R^d\times(\R^{d+1}\setminus0))$. Then $A$ is
parameter-dependent elliptic if $\sigma_\psi^\mu(A)$ is nowhere
vanishing. In the elliptic case, $A$ admits a parametrix, i.e., there
exists a $B\in \Psi_\cl^{-\mu}(\R^d;\R)$ such that
$AB-\textup{I},BA-\textup{I}\in \mathcal
S(\R_\tau;\Psi^{-\infty}(\R^d))$. This parametrix $B$ is essentially
unique, i.e., it is unique modulo $\mathcal
S(\R;\Psi^{-\infty}(\R^d))$. Moreover, $A(\tau)$ is invertible for
$|\tau|$ large and $B$ can be chosen to satisfy $B(\tau) =
A(\tau)^{-1}$ for $|\tau|$ large. Note also that
$\Psi_\cl^\mu(\R^d;\R)$ equipped with its canonical system of
seminorms is a nuclear Fr\'echet space.

\begin{ex} 
$A= -\Delta_y+\tau^2+\tau \in \Psi_\cl^2(\R^d;\R)$ has principal
  symbol $\sigma_\psi^2(A)(y,\eta,\tau) = |\eta|^2+\tau^2$ and is
  parameter-dependent elliptic.
\end{ex}


\subsubsection{Holomorphic Mellin symbols}\label{holom_Mellin}

Recall that we write $z\in\C$ as $z=\beta+\ii\tau$ with
$\beta,\tau\in\R$.

\begin{defn} \label{Def3.14}
For $\mu\in\R$, we define
\[ 
  \mathcal M^\mu(\R^d) = \mathcal H(\C;\Psi_\cl^\mu(\R^d)) \cap
  \cC^\infty(\R_\beta;\Psi_\cl^\mu(\R^d;\R_\tau)).
\]
as the set of holomorphic Mellin symbols of order $\mu$.
\end{defn} 

These Mellin symbols are entire functions of $z=\beta+i\tau$ taking
values in the nuclear Fr\'echet space $\Psi_\cl^\mu(\R^d)$ and are
also smooth functions of $\beta$ taking values in
$\Psi_\cl^\mu(\R^d;\R_\tau)$. Note that, as a consequence of the
Cauchy-Riemann equations, the principal symbol
$\sigma_\psi^\mu(h)(y,\eta,\tau)$ of $h\bigr|_{\Re
  z=\beta}\in\Psi_\cl^\mu(\R^d;\R_\tau)$ for $h\in\mathcal
M^\mu(\R^d)$ is independent of $\beta\in\R$.

\begin{prop}\label{hol_inv}
Let $a\in S^{(\mu)}(\R^d\times (\R^{d+1}\setminus0))$ be
elliptic. Then there exists a $h\in \mathcal M^\mu(\R^d)$ with
$\sigma_\psi^\mu(h)=a$ such that $h^{-1}\in \mathcal M^{-\mu}$
\end{prop}
\begin{proof}
This is proven as in \textsc{Witt} \cite{Wit2002}.
\end{proof}
  
\smallskip

The following result provides a means to control the action of
$\operatorname{op}_M(h)$ for $h\in \mathcal M^\mu(\R^d)$ on asymptotic
terms. Recall that $h^{(r)}(p)=\partial_z^r h(p) \in
\Psi_\cl^\mu(\R^d)$ for $p\in\C$.

\begin{prop}\label{important}
Let $h\in \mathcal M^\mu(\R^d)$, $(p,k)\in\C\times\N_0$, and $w\in
H^{s+\mu,\langle k\rangle}(\R^d)$ for some $s\in\R$. Then
\begin{equation}\label{impo}
  \varphi_0\operatorname{op}_M(h) \Gamma_{pk}w - \sum_{r=0}^k
  \frac1{r!}\,\Gamma_{p,k-r}\left[ h^{(r)}(p) w \right] \in
  \bigcap_{\epsilon>0} \cK^{s-\epsilon,1/2-\Re p +\epsilon}(\RR).
\end{equation}
\end{prop}
\begin{proof}
Let $m(z)$ be the Mellin transform of $\varphi(x) x^{-p}$. Then the
Mellin transform of $\varphi(x) x^{-p} \log^{k-r}\!x$ for $0\leq r\leq
k$ equals $m^{(k-r)}(z)$. Doing the computations modulo
$\bigcap_{\epsilon>0} \cK^{s-\epsilon,1/2-\Re p +\epsilon}(\RR)$,
we~find that (see Lemma~\ref{Lem3.8})
\[
\begin{aligned}
  \varphi_0\operatorname{op}_M(h) \Gamma_{pk}w &\equiv \frac{(-1)^k}{k!} \,
  \varphi_0\operatorname{op}_M(h) \left[\varphi(x)x^{-p} \log^k\! x
    \,w(y)\right] \\
  & = \frac{(-1)^k}{k!} \,\varphi_0 M^{-1} \left\{h(y,z,D_y) m^{(k)}(z) w(y)\right\}
  \\
  &\equiv \frac{(-1)^k}{k!} \, \varphi_0 M^{-1} \left\{ \sum\nolimits_{r=0}^k
  \frac1{r!}\,h^{(r)}(y,p,D_y) (z-p)^r m^{(k)}(z) w(y)\right\} \\
  & \equiv \varphi_0 \sum_{r=0}^k  \frac1{r!}\,\frac{(-1)^{k-r}}{(k-r)!}\, M^{-1} \left\{ 
  h^{(r)}(y,p,D_y) m^{(k-r)}(z) w(y)\right\} \\
  & \equiv \sum_{r=0}^k
  \frac1{r!}\,\Gamma_{p,k-r}\left[ h^{(r)}(p) w \right],
\end{aligned}
\]
where we have used that $ (z-p)^{k+1} m^{(k)}(z) \in \mathcal H(\C)
\cap \cC^\infty(\R_\beta;\R_\tau)$ and also that
\[
  \frac{(-1)^k}{k!} \, (z-p)^r m^{(k)}(z) - \frac{(-1)^{k-r}}{(k-r)!}
  \, m^{(k-r)}(z) \in \mathcal H(\C) \cap \cC^\infty(\R_\beta;\R_\tau)
\]
for $0\leq r\leq k$.
\end{proof}


\subsubsection{Cone-degenerate pseudodifferential operators}

In order to introduce cone-degenerate pseudodifferential operators, we
choose cut-off functions $\varphi,\varphi_0,\varphi_1\in
\cC_\cc^\infty(\overline\R_+)$ that localize near $x=0$ and satisfy
$\varphi\varphi_0=\varphi$, $\varphi\varphi_1=\varphi_1$.
 
\begin{defn} \label{Def3.15} 
For $\mu\in\R$, the class $\Psi_\cc^\mu(\overline\R_+^{1+d})$ of
cone-degenerate pseudodifferential operators on $\R_+^{1+d}$ consists
of all pseudodifferential operators $A$ on $\RR$ which are of the form
\begin{equation}\label{cdpsdo}
  A =  A_M + A_\psi + A_{r},
\end{equation}
where
\begin{enumerate}[(i)]

\item  $A_M = \varphi \operatorname{op}_M(h) \varphi_0$ for some
  $h=h(x,z,y,D_y)\in \cC^\infty(\overline\R_+;\mathcal M^\mu(\R^d))$,

\item $A_\psi =
  (1-\varphi)\operatorname{op}_\psi(a)(1-\varphi_1)$ for some $a\in
  S_\cl^\mu(\overline\R_+^{1+d}\times\R^{1+d})$.

\item $A_{r}$ has integral kernel in $\mathscr
  S_0(\overline\R_+^{1+d})\hat{\otimes} \mathscr
  S_0(\overline\R_+^{1+d})$ (with respect to the measure
  $\psi^{-2\delta}(x)\,\dd x
  \dd y$ attached to the right factor).
\end{enumerate}
\end{defn}

Notice that any operator $A$ of the form in \eqref{cdpsdo} is a
pseudodifferential operator on $\RR$. The point of this definition is
to enforce control on the behavior as $x\to+0$ in a specific way (see,
e.g., the mapping properties in Proposition~\ref{prop3.20} below).

\begin{rem}
Although the weight factor $\psi^{-2\delta}(x)$ appears explicitly in
the definition of the residual class
$\Psi_\cc^{-\infty}(\overline\R_+^{1+d})$ consisting of the operators
$A_r$ in \eqref{cdpsdo}, the classes
$\Psi_\cc^\mu(\overline\R_+^{1+d})$ and
$\Psi_\cc^{-\infty}(\overline\R_+^{1+d})$ are in fact independent of
$\delta\in\R$.
\end{rem}


\subsubsection{Symbolic structure}

For the rest of this section, we develop certain elements of the
calculus of pseudodifferential operators in
$\Psi_\cc^\infty(\overline\R_+^{1+d}) =
\bigcup_{\mu\in\R}\Psi_\cc^\mu(\overline\R_+^{1+d})$. We start with
the symbolic structure.

First note that the vector fields $a(x,y)x\partial_x +
\sum_{j=1}^d a_j(x,y)\partial_j$ tangent to $\partial\oRR$, where
$a,a_j\in \cC^\infty(\oRR)$,
are the $\cC^\infty$ sections of a vector bundle over $\oRR$ that we
denote by $\widetilde T^*\oRR$ and call it the compressed cotangent
bundle (see, e.g., \textsc{Melrose} \cite{Mel1993}). Indeed, the
covariable to $(x,y)\in\oRR$ in $\widetilde T^*\oRR$ can be taken to
be $(\tilde\xi,\eta)$ with $\tilde\xi=\psi(x)\xi$.

\begin{defn}
The compressed principle symbol $\tilde\sigma_\psi^\mu(A) \in
S^{(\mu)}(\widetilde T^\ast\overline\R_+^{1+d}\setminus0)$ of an operator
$A\in \Psi_\cc^\mu(\overline\R_+^{1+d})$ is defined as
\[
  \tilde \sigma_\psi^\mu(A)(x,y,\tilde\xi,\eta) = \varphi(x)\,
  \sigma_\psi^\mu(h)(x,y,\tau,\eta)\bigr|_{\tau=-\tilde\xi} +
  \left(1-\varphi(x)\right) \sigma_\psi^\mu(A_\psi)(x,y,\xi,\eta).
\]
Here, $\sigma_\psi^\mu(h)$ is the parameter-dependent principal symbol
of $h = h(x,y,z,D_y)$.
\end{defn}

\begin{prop}
The short sequence
\[
  0 \longrightarrow \Psi_\cc^{\mu-1}(\overline\R_+^{1+d})
  \longrightarrow \Psi_\cc^\mu(\overline\R_+^{1+d})
  \xrightarrow{\tilde\sigma_\psi^\mu} S^{(\mu)}(\widetilde
  T^\ast\overline\R_+^{1+d}\setminus0) \longrightarrow 0
\]
is split exact.
\end{prop}

Consequently, the compressed principal symbol
$\tilde\sigma_\psi^\mu(A)$ provides control on operators in
$\Psi_\cc^\mu(\overline\R_+^{1+d})$ up to lower-order perturbations.

\smallskip

Still, control of the asymptotic behavior as $x\to+0$ is achieved with
the help of the full sequence
$\left(\sigma_\cc^{-j}(A)\right)_{j\in\N_0}$ of conormal symbols.

\smallskip

\begin{defn} \label{Def3.16}     
For $A\in \Psi_\cc^\mu(\overline\R_+^{1+d})$ written as in
\eqref{cdpsdo}, the conormal symbol $\sigma_\cc^{-j}(A)$ of conormal order $-j$ for
$j\in\N_0$ is defined as
\[
  \sigma_\cc^{-j}(A)(z) = \frac1{j!}\,\partial_x^j h(0, y, z, D_y) \in
  \mathcal M^\mu(\R^d).
\]
\end{defn}

\begin{lem}
There is a compatibility condition between $\tilde \sigma_\psi^\mu(A)$
and $\sigma_\cc^0(A)$, namely
\begin{equation}\label{ccond}
  \tilde \sigma_\psi^\mu(A)(0,y,\tilde\xi,\eta)=
  \sigma_\psi^\mu(\sigma_\cc^0(A))(y,\eta,\tau)\bigr|_{\tau=-\tilde\xi}.
\end{equation}
\end{lem}
It is this compatibility condition which later will guarantee that the
governing equations for the coefficents $\gamma_{pk}u$ are
symmetrizable hyperbolic.

\begin{rem}
In order to provide a heuristic explanation of how control on the
asymptotic behavior as $x\to+0$ is achieved by the sequence
$\left(\sigma_\cc^{-j}(A)\right)_{j\in\N_0}$ notice that, informally,
we have that
\begin{equation}\label{heuristic}
  A \sim \sum_{j\geq0} x^j \operatorname{op}_M(\sigma_\cc^{-j}(A)(z))
  \quad \text{as $x\to +0$}
\end{equation}
upon performing a Taylor series expansion of $A_M$ at $x=0$. (See also
Proposition~\ref{ae3}.)
\end{rem}

\begin{ex}
The first-order differential operator
\[
   A = a(x,y)\psi(x) D_x + \sum_{j=1}^d a_j(x,y) D_j + b(x,y),
\]
where $a, a_j, b \in \cC^\infty_b(\overline\R_+^{1+d})$, belongs
to $\Psi_\cc^1(\oRR)$.  Then it is readily checked that
\begin{enumerate}[(i)]
\item $\tilde \sigma_\psi^1(A)(x,y,\tilde\xi,\eta) = a(x,y)\tilde\xi +
  \sum_{j=1}^d a_j(x,y)\eta_j$.

\item $\sigma_\cc^0(A)(z) = \ii a(0,y) z + \sum_{j=1}^d a_j(0,y)D_j+b(0,y)$.

\item $\sigma_\psi^1(\sigma_\cc^0(A))(y,\eta,\tau) = -\,a(0,y)\tau +
  \sum_{j=1}^d a_j(0,y)\eta_j$.
\end{enumerate}
\end{ex}


\subsubsection{Compositions and adjoints}

Now we get into the calculus of cone-degenerate pseudodifferential
operators. First consider their formal adjoints. Let $A^*$ be formal
adjoint operator of $A \in \Psi_\cc^\mu(\overline\R_+^{1+d})$, i.e.,
\[
  \langle Au,v \rangle = \langle u, A^*v\rangle, 
  \quad u,v\in\mathscr C_\cc^\infty(\R_+^{1+d}).
\]

\begin{prop} 
The class $\Psi_\cc^\infty(\overline\R_+^{1+d})$ of
cone-degenerate pseudodifferential operators is closed under taking
adjoints in the sense that whenever $A\in
\Psi_\cc^\mu(\overline\R_+^{1+d})$, then $A^* \in
\Psi_\cc^\mu(\overline\R_+^{1+d})$. Moreover,
\begin{enumerate}[{\rm (i)}]
\item $\tilde\sigma_\psi^\mu(A^*) = \tilde\sigma_\psi^\mu(A)^\ast$
\smallskip 

\item $\sigma_\cc^0(A^*)(z) = \sigma_\cc^0(A)(1-2\delta-\bar{z})^*$.
\end{enumerate}
\end{prop}

Another result is that cone-degenerate pseudodifferential operators are closed under compositions.

\begin{prop}
The class $\Psi_\cc^\infty(\overline\R_+^{1+d})$ of cone-degenerate
pseudodifferential operators is closed under compositions in the
sense that whenever $A\in\Psi_\cc^\mu(\overline\R_+^{1+d})$ and $B\in
\Psi_\cc^\nu(\overline\R_+^{1+d})$, then their composition $A\circ B$
belongs to $ \Psi_\cc^{\mu+\nu}(\overline\R_+^{1+d})$. Moreover,
\begin{enumerate}[{\rm (i)}]
\item $\tilde\sigma_\psi^{\mu+\nu}(A\circ B)= \tilde\sigma_\psi^\mu(A) \,
  \tilde\sigma_\psi^\nu(B)$,
\smallskip 

\item $\sigma_\cc^{-\ell}(A\circ B)(z) = \sum_{j+k=\ell}
  \sigma_\cc^{-j}(A)( z-k) \, \sigma_\cc^{-k}(B)(z)$ \/ for
  $\ell\in\N_0$.
\end{enumerate}
\end{prop}


\subsubsection{Mapping properties}

Cone-degenerate pseudodifferential operators act continuously in the
scale of Sobolev spaces with asymptotics. To see this, we start with
the following result:

\begin{lem}
For each $\mu\in\R$, $\Psi_\cc^\mu(\overline\R_+^{1+d})\subset
\bigcap_P\mathcal L(\mathscr S_P(\overline\R_+^{1+d}))$, where the
intersection is over all asymptotic types $P$.
\end{lem}
\begin{proof}
This is a standard result in the theory of cone-degenerate
pseudodifferential operators. In particular, it relies on the fact
that the conormal symbols are assumed to be holomorphic.
\end{proof}
  
\begin{prop} \label{prop3.20}
For $\mu \in \R$, 
\[
  \Psi_\cc^\mu(\overline\R_+^{1+d})\subset \bigcap_{s,P,\theta}
  \mathcal L\bigl(H_{P,\theta}^{s+\mu,\delta}(\overline\R_+^{1+d}),
  H_{P,\theta}^{s,\delta}(\overline\R_+^{1+d})\bigr).
\]
where intersection is over all $s\in\R$, 
$P\in\underline{\operatorname{As}}^\delta$, and $\theta\geq0$.

\end{prop}
\begin{proof}
It is well-known that $\Psi_\cc^\mu(\overline\R_+^{1+d})\subset
\bigcap_{s,\gamma} \cL\bigl(\cK^{s+\mu,\gamma}(\RR),
\cK^{s,\gamma}(\RR)\bigr)$. That this holds for all $\gamma\in\R$ is
again a consequence of the holomorphy of the conormal symbols.

By the closed graph theorem and complex interpolation, it is enough to
show that
\begin{equation}\label{paula}
  \varphi_0 \operatorname{op}_M(h)\Gamma_{pk}w \in
  H_{P,\theta}^{s,\delta}(\overline\R_+^{1+d})
\end{equation}
whenever $(p,k)\in P$, $w\in H^{s+\mu+\Re p+\delta-1/2,\langle
  k\rangle}(\R^d)$, $1/2-\delta-\theta < \Re p$, $\Re
p-(1/2-\delta-\theta) \notin\N$, and $h\in
\cC^\infty(\overline\R_+;\mathcal M^\mu(\R^d))$, where $\Gamma_{pk}w$
is given in~\eqref{nmu}. Let $\kappa$ be the smallest integer such
that $\Re p-\kappa<1/2-\delta-\theta$. Writing $h(x,z) = \sum_{0\leq
  j<\kappa} x^j h_j(z) + x^\kappa h'(x,z)$, where $h_j\in\mathcal
M^\mu(\R^d)$ for $0\leq j<\kappa$ and
$h'\in\cC^\infty(\overline\R_+;\mathcal M^\mu(\R^d))$, one has
\[
  \varphi_0(x) x^j \operatorname{op}_M(h_j) \Gamma_{pk}w =
  \sum_{r=0}^k \frac1{r!}\,\Gamma_{p-j,k-r}\bigl[h_j^{(r)}(p)w\bigr] + v_j,
\]
where $v_j\in \cK^{s-\theta+j,\delta+\theta}(\RR)$, and
\[
  \varphi_0(x) x^\kappa \operatorname{op}_M(h') \Gamma_{pk}w \in
  \cK^{s-\theta+\kappa,\delta+\theta}(\RR).
\]
Altogether, \eqref{paula} follows. This completes the proof.
\end{proof}

As a consequence, $\gamma_{pk}(Au)$ for $u\in
H_{P,\theta}^{s+\mu,\delta}(\overline\R_+^{1+d})$ is 
computable in terms of $\{\sigma_\cc^{-j}(A)\}_{j\geq0}$.

\begin{prop}\label{ae3}
Let $A\in \Psi_\cc^\mu(\overline\R_+^{1+d})$, $u\in
H_{P,\theta}^{s+\mu,\delta}(\overline\R_+^{1+d})$, and $(p,k)\in
P$, where $\Re p>1/2-\delta-\theta$. Then
\begin{equation}\label{asymp_exp}
  \gamma_{pk}(Au) = \sum_{j\geq0}\sum_{\ell-r=k}
  \frac1{r!}\,\partial_z^r\sigma_\cc^{-j}(A)(p+j)\gamma_{p+j,\ell}(u).
\end{equation}
The sum in the right-hand side is over those $(j,\ell,r)$,
where  $\Re p+j<1/2-\delta$ and  $0\leq \ell<m_{p+j}$. In
particular, this sum is finite.
\end{prop}

\begin{proof}
By Lemma~\ref{dense} and continuity of the trace maps according to
Proposition~\ref{trace_thm}, it is enough to verify \eqref{asymp_exp}
when $u\in\cS_P(\oRR)$. In this case, \eqref{asymp_exp} follows from
\eqref{defnsp}, \eqref{heuristic}, \eqref{opmhh}, \eqref{paul}. (See
\cite{LW2004} for such explicit calculations.)
\end{proof}


\subsubsection{Further results}

Here we collect several results about the calculus for cone-degenerate
pseudodifferential operators that we will need later, e.g., when
constructing a symmetrizer.

It is crucial that an integration by parts produces no boundary
terms. It is precisely this property which allows us to treat the
initial-boundary value problem Eq.~\eqref{ibvp} as a Cauchy problem.

\begin{lem}\label{ibparts}
For $A\in\Psi_\cc^1(\oRR;\C^N)$ and $u,v\in \cK^{1,\delta}(\RR;\C^N)$,
it holds that
\begin{equation}\label{fgt}
  \left\langle Au,v\right\rangle = \left\langle u,A^* v\right\rangle.
\end{equation}
\end{lem}
\begin{proof}
Property \eqref{fgt} holds whenever
$u,v\in\cC_\cc^\infty(\RR;\C^N)$. Because $\cC_\cc^\infty(\RR;\C^N)$
is dense in $\cK^{1,\delta}(\RR;\C^N)$, the result follows.
\end{proof}

The existence of so-called order reductions is assured next.

\begin{prop}\label{Lem4.3}
Let $\mu\in\R$. Then there exists a selfadjoint, positive definite
operator $\Lambda^\mu\in\Psi_\cc^\mu(\overline\R_+^{1+d})$ such that
$\Lambda^{-\mu}=(\Lambda^\mu)^{-1}\in
\Psi_\cc^{-\mu}(\overline\R_+^{1+d})$. In particular,
\begin{equation}\label{or}
  \Lambda^\mu\colon
  H_{P,\theta}^{s+\mu,\delta}(\overline\R_+^{1+d}) \to
  H_{P,\theta}^{s,\delta}(\overline\R_+^{1+d})
\end{equation}
is an isomorphism for all $s\in\R$,
$P\in\underline{\textup{As}}^\delta$, and $\theta\geq0$.
\end{prop}
\begin{proof}
One way to prove the result is to start with a parameter-dependent
version $\Psi_\cc^{\mu/2}(\overline\R_+^{1+d};\R)$ of the class
$\Psi_\cc^{\mu/2}(\overline\R_+^{1+d})$ consisting of families $A=
\left(A(\lambda)\right)_{\lambda\in\R}\subset
\Psi_\cc^{\mu/2}(\overline\R_+^{1+d})$, as in
Section~\ref{parameter-dep}. Choose a parameter-elliptic family $A\in
\Psi_\cc^{\mu/2}(\overline\R_+^{1+d};\R)$ with compressed principal
symbol $\tilde\sigma_\psi^{\mu/2}(A) = (\tilde\xi^2 + |\eta|^2 +
\lambda^2)^{\mu/4}$ and leading conormal symbol $\sigma_\cc^0(A) =
\sigma_\cc^0(A)(z,\lambda)\in \mathcal M^{\mu/2}(\R^d;\R)$ such that
$\sigma_\cc^0(A)^{-1}\in \mathcal M^{-\mu/2}(\R^d;\R)$ (see
Proposition~\ref{hol_inv}). Then $A(\lambda) \in
\Psi_\cc^{\mu/2}(\overline\R_+^{1+d})$ is invertible for $|\lambda|
\gtrsim 1$, with $A(\lambda)^{-1} \in
\Psi_\cc^{-\mu/2}(\overline\R_+^{1+d})$. Now pick a $\lambda\in \R$
with $|\lambda|$ large and set $\Lambda^\mu = A(\lambda)^*A(\lambda)$.
\end{proof}

We have the following form of G\aa rding's inequality.
 
\begin{prop} \label{Prop3.21}
Let $\mu\geq0$. Suppose that $A \in \Psi_\cc^{\mu}(\oRR;\C^N)$ has
a positive definite compressed principle symbol satisfying
$\tilde\sigma_\psi^\mu(A)(x,y,\tilde\xi,\eta)\gtrsim
(\tilde\xi^{\,2}+|\eta|^2)^{\mu/2}\,\textup{I}_{N}$. Then there exists
a $C=C^* \in \Psi_\cc^{\mu-1}(\overline\R_+^{1+d};\C^N)$ and a
constant $c>0$ such that
\[
  \Re \left\langle Au, u \right\rangle \ge c \left\|
  u\right\|_{\cK^{\mu/2,\delta}}^2 - \bigl\langle Cu, u \bigr\rangle
\]
for all $u\in\cK^{\mu,\delta}(\RR;\C^N)$.	
\end{prop}
\begin{proof}
Writing $u = \Lambda^{-\mu/2}v$ with
$v\in\cK^{\mu/2,\delta}(\RR;\C^N)$, we can assume that $\mu=0$. Then
$\tilde\sigma_\psi^0(A)(x,y,\tilde\xi,\eta)\geq 2c\, \textup{I}_{N}$
for some constant $c>0$. Choose $B\in \Psi_\cc^0(\oRR;\C^N)$ such
that $\tilde\sigma_\psi^0(B) = \bigl(\tilde\sigma_\psi^0(A) - c\,
\textup{I}_{N}\bigr)^{1/2}$. By construction, $C= B^*B -\Re A + c \in
\Psi_\cc^{-1}(\oRR;\C^N)$. We obtain
\[
  \Re \left\langle Au, u \right\rangle = \left\langle \left(\Re A\right)u, u
  \right\rangle = c\,\|u\|^2 + \|Bu\|^2 - \bigl\langle Cu, u
  \bigr\rangle \geq c\,\|u\|^2 - \bigl\langle Cu, u \bigr\rangle
\]
for $u\in\cK^{0,\delta}(\RR;\C^N)$.
\end{proof}
  

\section{Proof of the main results}\label{PMR}

In this section, we establish our main results. In fact, it suffices
to prove Theorem~\ref{thm2}. Theorem~\ref{thm1} is a special case of
Theorem~\ref{thm2}, where $s\geq0$, $\delta=0$, $P=P_0$, and $\theta_r
= s-r+\sigma$ for $0\leq r\leq \sigma$. Besides, Theorem~\ref{thm1}
has been proven independently in
Section~\ref{sec2}. Theorem~\ref{thm3} follows from Theorem~\ref{thm2}
in the usual way using coordinate invariance (see \cite{LRW2022}) and
finite propagation speed. An alternative argument retraces the steps
of the proof of Theorem~\ref{thm2} displayed below and uses
cone-degenerate pseudodifferential operators from a class
$\Psi_\cc^\mu(\overline{\Omega};\C^N)$, where now operators in this
class are additionally assumed to be properly supported.

\medskip

We consider the Cauchy problem
\begin{equation}\label{psdo_eq}
\left\{ \enspace
\begin{aligned}
  \p_t u + \cA(t,x,y,xD_x,D_y) u &= f(t,x,y), \quad (t,x,y)\in (0,T)\times\RR, \\
  u\bigr|_{t=0} &= u_0(x,y),
\end{aligned}
\right.
\end{equation}
where
\[
   \cA \in \cC^\infty\bigl([0,T]; \Psi_\cc^1(\overline\R_+^{1+d};\C^N) 
   \bigr).
\]
We assume that the operator $\p_t + \cA(t,x,y,xD_x,D_y)$ is hyperbolic
in the sense that $\cA$ admits a \textit{symbolic symmetrizer\/}. This
means that there exists a $b \in
\cC^\infty\bigl([0,T];S^{(0)}(\widetilde
T^*\overline\R_+^{1+d}\setminus0 ;M_{N\times N}(\C))\bigr)$ such that
\begin{enumerate}[(i)]
\item $b(t,x,y,\tilde\xi,\eta)=b(t,x,y,\tilde\xi,\eta)^* \geq
  c\,\operatorname{I}_{N}$ for some constant $c>0$,
\item $b(t, x, y, \tilde\xi,\eta) \tilde\sigma_\psi^1(\mathcal
  A)(t,x,y,\tilde\xi,\eta)$ is skew-Hermitian for all $(t, x, y,
  \tilde\xi, \eta)$.
\end{enumerate}

\begin{ex} 
The operator $\cA(t,x,y,xD_x,D_y) = x A(t,x,y)\partial_x +
\sum_{j=1}^d A_j(t,x,y)\partial_j + B(t,x,y)$ from Eq.~\eqref{ibvp2}
was assumed to satisfy these assumptions.
\end{ex}

      
\subsection{Well-posedness in weighted Sobolev spaces}

Employing the symbolic symmetrizer, $b$, we first construct a genuine
symmetrizer, $\cB$.

\begin{lem} \label{Lem4.1}
Let $b\in \cC^\infty\bigl([0,T];S^{(0)}(\widetilde
T^*\overline\R_+^{1+d}\setminus0;$ $M_{N\times N}(\C))\bigr)$ be a
symbolic symmetrizer for $\cA$. Then there exists a $\mathcal B\in
\cC^\infty\bigl([0,T];\Psi_\cc^0(\overline\R_+^{1+d}; \C^N)\bigr)$
such that
\[
  \tilde \sigma_\psi^0(\cB(t))=b(t) \quad \textup{and} \quad
  \cB(t)=\cB(t)^* \geq c \, \textup{I}_{N}
\]
for some $c>0$ and all $t \in [0, T]$.
\end{lem}
 
\begin{proof} 
We pick a $\cB_1\in \cC^\infty\bigl([0,T];
\Psi_\cc^0(\overline\R_+^{1+d};\C^N)\bigr)$ with
$\tilde\sigma_\psi^0(\cB_1(t))=b(t)$ for $t \in [0, T]$ and set $\cB_0
=\left(\cB_1+\cB_1^*\right)/2$. Then
$\tilde\sigma_\psi^0(\cB_0(t))=b(t)$ for $t \in [0,T]$. By
Proposition~\ref{Prop3.21}, there exists a $\mathcal C\in
\cC^\infty\bigl([0,T];$ $\Psi_\cc^{-1}(\overline\R_+^{1+d};
\C^N)\bigr)$ with $\mathcal C(t)=\mathcal C(t)^*$ for $t\in[0,T]$ such
that, for any $u \in \cC^\infty_c(\R^{1+d}_+;\C^N)$ and $t \in [0,
  T]$,
\[
  \bigl\langle \cB_0(t) u, u \bigr\rangle \geq
c\,\|u\|^2 - \bigl\langle \mathcal C(t) u, u \bigr\rangle. 
\]
It follows that the operator $\cB=\cB_0 + \mathcal C$
has the desired properties.
 \end{proof}
 
Next we derive energy estimates for Eq.~\eqref{psdo_eq} in the
weighted Sobolev spaces $\cK^{s,\delta}(\RR;\C^N)$.

We start with the case $s=0$. As usual, the proof of the next
proposition relies on the following facts (as was already mentioned in
the introduction):
\begin{itemize}
\item $\langle \cB(t) u, u\rangle$ is equivalent to $\|u\|^2$
  uniformly in $t\in[0,T]$,
\item Integration by parts produces no boundary terms (see
  Lemma~\ref{ibparts}),
\item $\cB\cA +(\cB\cA)^* \in
  \cC^\infty([0,T];\Psi_\cc^0(\oRR;\C^N))$.
\end{itemize}

\begin{prop}\label{Lem4.2}
Let $u\in \cC([0,T];\mathcal K^{1,\delta}(\R_+^{1+d};\C^N)) \cap
\cC^1([0,T];\mathcal K^{0,\delta}(\R_+^{1+d};\C^N)) $. Then
\begin{equation}\label{4.2}
  \sup_{0\leq t\leq T}\|u(t)\| \lesssim \|u(0)\| + \int_0^T
  \|\partial_t u(t) - \cA(t)u(t)\|\,\dd t.
\end{equation}
\end{prop}

\begin{proof}
Let $u(0)=u_0$, $\partial_t u -\cA u=f$. By construction,
there exists a constant $C>0$ such that
\[
  2\Re(\cB\mathcal A)+\partial_t\cB\leq 2C\cB.
\]
Then
\[
\begin{aligned}
  \partial_t
  \bigl( \langle \cB u,u\rangle \ee^{-2C t}\bigr) &\leq   \partial_t
  \bigl( \langle \cB u,u\rangle \ee^{-2C t}\bigr) - \bigl(
  2\Re \langle \cB\cA u,u\rangle - 2C \langle\cB
  u,u\rangle + \langle (\partial_t\cB)u,u\rangle
  \bigr)\ee^{-2C t} \\ &= 2\Re\langle\cB f,u\rangle \ee^{-2C t}.
\end{aligned}
\]
Setting $K=\sup_{t\in[0,T]}\langle \cB(t)u(t),u(t)\rangle^{1/2}\ee^{-C
  t}$, the Cauchy-Schwarz inequality implies that
\[
  \langle \cB(t) u(t),u(t)\rangle \ee^{-2C t} \leq \langle
  \cB(0) u_0,u_0\rangle + 2 K \int_0^t \langle \cB(s)
  f(s),f(s)\rangle^{1/2}\ee^{-C s}\,\dd s,
\]
i.e.,
\[
  \Bigl(K-\int_0^T \langle \cB(t) f(t),f(t)\rangle^{1/2}\ee^{-C
    t}\,\dd t\Bigr)^2 \leq \Bigl(\langle \cB(0) u_0,u_0\rangle^{1/2} +
  \int_0^T \langle \cB(t) f(t),f(t)\rangle^{1/2}\ee^{-C t}\,\dd
  t\Bigr)^2
\]
It follows that
\[
  K \leq \langle \cB(0) u_0,u_0\rangle^{1/2} + 2\int_0^T \langle
  \cB(t) f(t),f(t)\rangle^{1/2}\ee^{-C t}\,\dd t.
\]
Because the norm $v\mapsto \langle \cB(t)v,v\rangle^{1/2}\ee^{-C t}$
is equivalent to $\|v\|$ uniformly in $t\in[0,T]$, this finishes the
proof.
\end{proof} 

As an immediate consequence we have the next result.

\begin{prop} \label{Thm4.4}
Let $u\in \cC([0,T];\mathcal K^{s+1,\delta}(\R_+^{1+d};\C^N)) \cap
\cC^1(([0,T];\mathcal K^{s,\delta}(\R_+^{1+d};\C^N))$ for some
$s\in\R$. Then
\begin{equation}\label{estss}
  \sup_{0\leq t\leq T}\|u(t)\|_{\mathcal K^{s,\delta}} \lesssim
  \|u(0)\|_{\mathcal K^{s,\delta}} + \int_0^T \|\partial_t u(t)
  -\cA(t) u(t)\|_{\mathcal K^{s,\delta}}\,\dd t.
\end{equation}
\end{prop}

\begin{proof}
Let again $u(0)=u_0$, $\partial_t u(t) - \cA u=f$. Let
$\Lambda^s\in\Psi_\cc^s(\overline\R_+^{1+d})$ be a scalar invertible
operator such that
$(\Lambda^s)^{-1}\in\Psi_\cc^{-s}(\overline\R_+^{1+d})$, as
constructed in Lemma~\ref{Lem4.3}. Then, $\Lambda^s u$ solves the
system
\begin{equation} \label{4.3}
\left\{ \enspace
\begin{aligned}
   & \partial_t (\Lambda^s u) + \Lambda^s \mathcal
  A(t)\Lambda^{-s}(\Lambda^s u) = \Lambda^s f(t),
  \\
  & (\Lambda^s u)\bigr|_{t=0} = \Lambda^s u_0.
\end{aligned}
\right.
\end{equation}
Notice that $\Lambda^s \cA \,\Lambda^{-s} \in \cC^\infty([0,T];
\Psi_\cc^1(\overline\R_+^{1+d};\C^N))$ and $\tilde
\sigma_\psi^1(\Lambda^s\cA\Lambda^{-s})= \tilde
\sigma_\psi^1(\cA)$. Hence, system \eqref{4.3} is sym\-me\-tri\-zable
hyperbolic. Applying Proposition~\ref{Lem4.2} yields
\[
  \sup_{0\leq t\leq T}\|\Lambda^s u(t)\| \lesssim \|\Lambda^s u_0\| +
  \int_0^T \|\Lambda^s f(t)\|\,\dd t.
\]
As $\|\Lambda^s\cdot\|$ is an equivalent norm on $\mathcal
K^{s,\delta}$, we obtain estimate \eqref{estss}.
\end{proof}

\begin{prop} \label{Thm1.1}
Let $u_0\in \mathcal K^{s,\delta}(\R_+^{1+d};\C^N)$ and $f\in
L^1((0,T); \mathcal K^{s,\delta}(\R_+^{1+d};\C^N))$ for some
$s\in\R$. Then \textup{Eq.~\eqref{psdo_eq}} possesses a unique
solution $u\in \cC([0,T]; \mathcal
K^{s,\delta}(\R_+^{1+d};\C^N))$. Moreover, the energy inequality
\[
    \sup_{0\leq t\leq T}\|u(t)\|_{\mathcal K^{s,\delta}} \lesssim
    \|u_0\|_{\mathcal K^{s,\delta}} + \int_0^T \|f(t)\|_{\mathcal
      K^{s,\delta}}\,\dd t 
\]
holds.
\end{prop}
\begin{proof}
\textsl{Uniqueness}. Let $u_0=0$, $f=0$. Then $\cA(t)u \in \cC([0,T];
\mathcal K^{s-1,\delta}(\R_+^{1+d};\C^N))$ and, consequently, we
obtain $\partial_t u \in \cC([0,T]; \mathcal
K^{s-1,\delta}(\R_+^{1+d};\C^N))$ from the equation. Hence,
estimate~\eqref{estss} (with $s$ replaced with $s-1$) yields $u=0$.

\smallskip

\textsl{Existence}.  We argue by duality. Set $\mathcal Y=\{v\in
\cC([0,T];\cS(\R^d;\C^N))\mid v(T)=0\}$. The operator $\partial_t +
\cA(T-t)^*$ is symmetrizable hyperbolic. Hence, estimate~\eqref{estss}
implies (after the change of variables $t\mapsto T-t$)
\[
  \sup_{0\leq t\leq T}\|v(t)\|_{\cK^{-s,\delta}} \lesssim \int_0^T
  \|-\partial_t v(t) + \cA(t)^* v(t)\|_{\cK^{-s,\delta}}\,\dd t, \quad
  v\in\mathcal Y.
\]
We now consider the functional
\begin{equation}\label{BS}
  g \mapsto \int_0^T \left\langle f(t),v(t)\right\rangle\dd t +
  \left\langle u_0,v(0)\right\rangle 
\end{equation}
on the space $\left(-\partial_t + \cA(t)^*\right)\mathcal Y$, where $g
= -\partial_t v + \cA(t)^* v$, $v\in\mathcal Y$. We have the estimate
\[
\begin{aligned}
  \left| \int_0^T \left\langle f(t),v(t)\right\rangle\dd t +
  \left\langle u_0,v(0)\right\rangle\right| &\leq
  \|f\|_{L_t^1\cK^{s,\delta}} \|v\|_{L_t^\infty\cK^{-s,\delta}}
  +\|u_0\|_{\cK^{s,\delta}} \|v(0)\|_{\cK^{-s,\delta}} \\ &\lesssim \int_0^T
  \|g(t)\|_{\cK^{-s,\delta}}\,\dd t.
\end{aligned}
\]
By the Hahn-Banach theorem, the functional in \eqref{BS} extends to a
bounded functional on the space $L^1((0,T);
\cK^{-s,\delta}(\RR;\C^N))$. By duality, such an extension is given as
$g \mapsto \int_0^T \langle u(t),g(t)\rangle\,\dd t$ for some uniquely
determined $u\in L^\infty((0,T);\cK^{s,\delta}(\RR;\C^N))$. We obtain
that
\begin{equation}\label{nvf}
  \int_0^T \langle u(t), -\partial_t v(t) +\cA(t)^* v(t)\rangle\,\dd t
  = \int_0^T \left\langle f(t),v(t)\right\rangle\dd t + \left\langle
  u_0,v(0)\right\rangle, \quad v\in \mathcal Y.
\end{equation}
Taking $v\in \cC_\cc^\infty((0,T)\times\RR;\C^N)$ demonstrates that
$u$ is a weak solution to $\partial_t u +\mathcal A(t) u = f(t)$ on
$(0,T)\times \RR$.

If $f\in L^1((0,T);\cK^{s+1,\delta}(\RR;\C^N))$, then $u\in
L^\infty((0,T);\cK^{s+1,\delta}(\RR;\C^N))$ and, moreover, $u \in \cC
([0,T];\cK^{s,\delta}(\RR;\C^N))$ from the equation. Indeed, $u$ is
absolutely continuous with values in $\cK^{s,\delta}(\RR;\C^N)$. In
addition, it follows from \eqref{nvf} that $u(0)=u_0$. In the general
case, we choose sequences $(u_{0m})\subset \cK^{s+1,\delta}(\RR;\C^N)$
and $(f_m) \subset \cC([0,T];\cK^{s+2,\delta}(\RR;\C^N))$ such that
\[
  u_{0m} \to u_0 \enspace \text{in $\cK^{s,\delta}(\RR;\C^N)$}, \quad
  f_m \to f \enspace \text{in $L^1((0,T);\cK^{s+1,\delta}(\RR;\C^N))$.}
\]
Let $(u_m)\subset \cC([0,T];\cK^{s+1,\delta}(\RR;\C^N)) \cap
\cC^1([0,T];\cK^{s,\delta}(\RR;\C^N))$ be the sequence of solutions to
Eq.~\eqref{psdo_eq}, with the data $(u_0,f)$ replaced with
$(u_{0m},f_m)$. By Proposition~\ref{Thm4.4}, $(u_m)$ is a Cauchy
sequence in $\cC([0,T];\cK^{s,\delta}(\RR;\C^N))$. It is readily seen
that its limit $u$ is the desired solution.
\end{proof}

Eventually, we discuss higher regularity with respect to $t$.

\begin{prop}\label{abcd}
Let $u_0\in \cK^{s+\sigma,\delta}(\RR;\C^N)$, $f\in
\bigcap_{r=0}^\sigma W^{r,1}((0,T);\cK^{s-r+\sigma}(\RR;\C^N))$ for
some $s\in\R$, $\sigma\in \N_0$. Then the unique solution $u$ to\/
\textup{Eq.~\eqref{psdo_eq}} belongs to the space
$\bigcap_{r=0}^\sigma \cC^r([0,T];$
$\cK^{s-r+\sigma}(\RR;\C^N))$. Moreover, the energy inequality
\[
    \sum_{r=0}^\sigma \sup_{0\leq t\leq T}\|\partial_t^r
    u(t)\|_{\mathcal K^{s-r+\sigma,\delta}} \lesssim \|u_0\|_{\mathcal
      K^{s+\sigma,\delta}} + \sum_{r=0}^\sigma\int_0^T \|\partial_t^r
    f(t)\|_{\mathcal K^{s-r+\sigma,\delta}}\,\dd t
\]
holds.
\end{prop}
\begin{proof}
We proceed by induction on $\sigma$.

The base case $\sigma=0$ was treated in Proposition~\ref{Thm1.1}.

For the induction step $\sigma\to \sigma+1$, suppose that $u_0\in
\cK^{s+\sigma+1,\delta}(\RR;\C^N)$, $f\in \bigcap_{r=0}^\sigma
W^{r,1}((0,T);$ $\cK^{s-r+\sigma+1}(\RR;\C^N))$. By induction
hypothesis, $u \in \bigcap_{r=0}^\sigma \cC^r([0,T];$
$\cK^{s-r+\sigma+1}(\RR;\C^N))$ (upon replacing $s$ with
$s+1$). Moreover, $u_t = \partial_t u$ solves the equation
\[
\left\{ \enspace
\begin{aligned}
  \partial_t(u_t) + \cA(t) u_t &= \partial_t f(t) - (\partial_t\cA)(t)
  u(t), \\ u_t(0) &= f(0) -\cA(0) u_0,
\end{aligned}
\right.
\]
where $f(0)-\cA(0) u_0\in \cK^{s+\sigma,\delta}(\RR;\C^N)$,
$\partial_t f - (\partial_t\cA) u\in \bigcap_{r=0}^\sigma
W^{r,1}((0,T);\cK^{s-r+\sigma}(\RR;\C^N))$. Again by induction
hypothesis, we conclude that $\partial_t u\in \bigcap_{r=0}^\sigma
\cC^r([0,T];\cK^{s-r+\sigma}(\RR;\C^N))$. Altogether, we obtain that
$u\in \bigcap_{r=0}^\sigma \cC^r([0,T];\cK^{s-r+\sigma+1}(\RR;\C^N))$
as required.
\end{proof}


\subsection{Well-posedness in Sobolev spaces with asymptotics}

Here we establish Theorem~\ref{thm2} in full generality. In fact, we
only derive the fundamental energy inequality. Then the rest of the
proof is completely analogous to the proof in the previous section,
and it is omitted.

\begin{prop}
Let $u\in \cC([0,T];H_{P,\theta}^{s+1,\delta}(\overline\R_+^{1+d};
\C^N))\cap \cC^1([0,T];H_{P,\theta}^{s,\delta}(\overline\R_+^{1+d};
\C^N))$ for some $P\in\underline{\textup{As}}^\delta$, $s\in\R$, and $\theta\geq0$. Then
\[
  \sup_{0\leq t\leq T} \|u(t)\|_{H_{P,\theta}^{s,\delta}} \lesssim
  \|u(0)\|_{H_{P,\theta}^{s,\delta}} + \int_0^T \|\partial_t u(t)+\cA(t)u(t)\|_{H_{P,\theta}^{s,\delta}} \,\dd t
\]  
\end{prop}

\begin{proof}
As before, we set $u_0=u(0)$ , $f= \partial_t u +\cA(t)u$.  By
interpolation, we may assume that $\pi_\C P\cap
\Gamma_{1/2-\delta-\theta}=\emptyset$. We then proceed in three steps.

\subsubsection*{Step 1} By Proposition~\ref{ae3}, taking traces one has
 \begin{equation}\label{taking_traces}
\left\{ \enspace
\begin{aligned}
 &  \p_t (\gamma_{pk}u) +  \sigma_\cc^0(\cA(t))(p)\gamma_{pk}u
  = \gamma_{pk}f - \sum_{\substack{j\geq0,\,\ell-r=k,\\(j,\ell,r)\neq(0,k,0)}}
  \frac1{r!}\,\partial_z^r\sigma_\cc^{-j}(\cA(t))(p+j)\,\gamma_{p+j,\ell}(u), \\
  & \gamma_{pk}u\bigr|_{t=0} = \gamma_{pk}u_0.
\end{aligned}
\right.
\end{equation}
This is a Cauchy problem for an $N\times N$ first-order hyperbolic
system in $(0,T)\times\R^d$. Hyperbolicity follows from
\[
  \sigma_\psi^1(\sigma_\cc^0(\cA(t))(p))(y,\eta) = \tilde\sigma_\psi^1(\cA(t))(0,y,0,\eta).
\]
Solving these systems successively using Proposition~\ref{NTD}, one finds
\begin{multline*}
  \sup_{0\leq t\leq T}\|\gamma_{pk}u(t)\|_{H^{s+\Re p +\delta
      -1/2,\langle k\rangle}} \\ \lesssim \sum_{j\geq0,\,\ell\geq k}
  \biggl(\|\gamma_{p+j,\ell}u_0\|_{H^{s+\Re p + j +
      \delta-1/2,\langle\ell\rangle}} + \int_0^T
  \|\gamma_{p+j,\ell}f(\tau)\|_{H^{s+\Re p +j
      +\delta-1/2,\langle\ell\rangle}} \,\dd \tau \biggr).
\end{multline*}
for $(p,k)\in P$, $\Re p>1/2-\delta+\theta$.

\subsubsection*{Step 2} Set $v_0 = u_0 - \sum_{\substack{(p,k)\in P,\\\Re
    p>1/2-\delta-\theta}} \Gamma_{pk}(
\gamma_{pk}u_0)\in\mathcal K^{s-\theta+1,\delta+\theta}(\R_+^{1+d};\C^N)$,
\[
  g = f - (\partial_t+\mathcal
  A(t))\left(\sum\nolimits_{\substack{(p,k)\in P,\\\Re p>1/2-\delta-\theta}}
  \Gamma_{pk} (\gamma_{pk}u)\right)\in
  \cC([0,T];\mathcal K^{s-\theta,\delta+\theta}(\R_+^{1+d};\C^N)).
\]
Now we solve the hyperbolic system
\[
\left\{ \enspace
\begin{aligned}
  & \partial_t v + \cA(t) v = g(t), \\
  & v\bigr|_{t=0} = v_0.
\end{aligned}
\right.
\]
Then, by Proposition~\ref{Thm1.1},
\[
 \sup_{0\leq t\leq T}
    \|v(t)\|_{\mathcal K^{s-\theta,\delta+\theta}} \lesssim
    \|v_0\|_{\mathcal K^{s-\theta,\delta+\theta}} + \int_0^T
    \|g(t)\|_{\mathcal K^{s-\theta,\delta+\theta}} \,\dd t.
\]

\subsubsection*{Step 3} Because of  $u=v + \sum_{\substack{(p,k)\in P,\\\Re p>1/2-\delta-\theta}}
  \Gamma_{pk} (\gamma_{pk}u)$, it follows that
\[
  \sup_{0\leq t\leq T} \|u(t)\|_{H_{P,\theta}^{s,\delta}} \lesssim
  \|u_0\|_{H_{P,\theta}^{s,\delta}} + \int_0^T
  \|f(t)\|_{H_{P,\theta}^{s,\delta}} \,\dd t.
\]
This completes the proof.
\end{proof}


\begin{rem}
By the arguments above, one can show that the Cauchy problem
\eqref{psdo_eq} is well-posed in the Sobolev spaces
$H_{P,\theta}^{s,(\delta,\rho)}(\oRR;\C^N)$, where
\begin{equation}\label{ppt3}
  H_{P,\theta}^{s,(\delta,\rho)}(\oRR) = \{u\mid \varphi u\in
  H_{P,\theta}^{s,\delta}(\oRR), (1-\varphi)u \in \langle
  x\rangle^{-\rho} H^s(\RR)\}
\end{equation}
for $s,\rho\in\R$, $P\in\As^\delta$, and $\theta\geq0$.
In view of
\[
  \cS_P(\oRR) = \bigcap_{s,\rho,\theta}
  H_{P,\theta}^{s,(\delta,\rho)}(\oRR),
\]
this immediately leads to the well-posedness of the Cauchy problem
\eqref{psdo_eq} in $\cS_P(\oRR;\C^N)$. More precisely,
Eq.~\eqref{psdo_eq} has a unique solution $u
\in\cC^\infty([0,T];\cS_P(\oRR;\C^N))$ provided that $u_0\in$ \linebreak
$\cS_P(\oRR;\C^N)$, $f\in\cC^\infty([0,T];\cS_P(\oRR;\C^N))$. 
\end{rem}




\appendix


\section{Some basic material}

For the reader's convenience, we collect here a few basic facts that are
used in the main body of the paper without further reference. (See
Section~\ref{notat} for the notation used.)


\subsection{The Mellin transform}\label{a1}

The Mellin transform $M$ is defined by
\[
  M u(z) = \tilde u(z) = \int_0^\infty x^{z-1}u(x)\,\dd x, \quad
  z\in\C,
\]  
for $u \in \cC_\cc^\infty(\R_+)$. It is then suitably
extended to other spaces of (generalized) functions. The inverse
transform is given by $M^{-1}v(x) = \frac1{2\pi\ii}
\int_{\Gamma_\beta} x^{-z}v(z)\,\dd z$ for a suitable $\beta\in\R$
depending on the situation under consideration.

Among others, the Mellin transform has the following properties:
\begin{enumerate}[(a)]

\item $\{-\,x \partial_x u\}\,\tilde{}\,(z) = z\,\tilde{u}(z)$,

\item $\{x^{-\gamma} u\}\,\tilde{}\,(z) = \tilde u(z-\gamma)$ for
  $\gamma\in\R$,

\item $\{\log x\,  u\}\,\tilde{}\,(z) = \partial_z \tilde u(z)$,
  
\item $M\colon L^2(\R_+,x^{-2\gamma}\,\dd x)\to
  L^2\bigl(\Gamma_{1/2-\gamma}, (2\pi\ii)^{-1}\dd z\bigr)$ is unitary
  for $\gamma\in\R$.
\end{enumerate}
In particular, for $h\in \mathcal M^\mu(\R^d)$,
$u\in\cH^{s,\gamma}(\RR)$, one has that
\begin{equation}\label{opmhh}
  \{\operatorname{op}_M(h)u\}\,\tilde{}\,(z) = h(z)\tilde u(z), \quad
  z \in \Gamma_{1/2-\gamma},
\end{equation}
and then that $\operatorname{op}_M(h) \colon \cH^{s+\mu,\gamma}(\RR)
\to \cH^{s,\gamma}(\RR)$ is continuous.

\begin{lem}\label{ppl}
Let $v\in\cS_P(\oRR)$ for some asymptotic type $P$. Then $\tilde
v(z,\cdot)$ is a meromorphic function of $z\in\C$ with values in
$\cS(\R^d)$ having poles at most at points $z=p$ for $p\in\pi_\C P$.
Moreover, for $v$ as given in \eqref{defnsp},
\begin{equation}\label{paul}
  \tilde v(z,\cdot) = \frac{v_{p,m_p-1}}{(z-p)^{m_p}} +
  \frac{v_{p,m_p-2}}{(z-p)^{m_p-1}} + \dotsc + \frac{v_{p0}}{z-p} + O(1)
  \quad \textup{as $z\to p$,}
\end{equation}
where $v_{p0},\dotsc,v_{p,m_p-1}\in \cS(\R^d)$. In addition, if $\chi\in
\cC^\infty(\C)$ satisfies $\chi(z)=0$ for
$\operatorname{dist}(z,\pi_\C P)\leq 1/2$ and $\chi(z)=1$ for
$\operatorname{dist}(z,\pi_\C P)\geq 1$, then $\chi \tilde v\in
\cC^\infty\bigl(\R_\beta;\cS(\R_{(y,\tau)}^{d+1})\bigr)$.
\end{lem}

Similar statements hold for $(\varphi v)\,\tilde{}\,(z,\cdot)$ when $v
\in H_{P,\theta}^{s,\delta}(\oRR)$, where now, however, $(\varphi
v)\,\tilde{}\,(z,\cdot)$ is holomorphic in the half-space $\{z\in\C\mid
\Re z\geq1/2-\delta\}$ and meromorphic in the open half-space
$\{z\in\C\mid \Re z>1/2-\delta- \theta\}$. As these statements are
more involved in their formulation and we do not make use of them, we
refrain from making these statements explicit. (See, e.g.,
\cite{RS1989, Sch1991}.)





\subsection{The hyperbolic Cauchy problem}

For the sake of completeness, we state a result about the
well-posedness of the hyperbolic Cauchy problem in the spaces
$H^{s,\langle k\rangle}(\R^d;\C^N)$ for $(s,k)\in\R\times\Z$. For
$k=0$, this is a standard result, but for $k\neq0$ we were not able to
locate it in the literature.

Let $\cB\in \cC^\infty([0,T];\Psi^1(\R^d;\C^N))$ and assume that the
operator $\partial_t +\cB(t,y,D_y)$ is symmetrizable hyperbolic
uniformly in $(t,y)\in[0,T]\times\R^d$ in the sense that there is a $b
\in S^{(0)}([0,T]\times\R^d\,\times$
$(\R^d\setminus0);\operatorname{Mat}_{N\times N}(\C))$ such that
\begin{enumerate}[(i)]
\item $b(t,y,\eta) = b(t,y,\eta)^*\geq c\,\textup{I}_{N}$ for
  some constant $c>0$,

\item $b(t,y,\eta)\,\sigma_\psi^1(\cB)(t,y,\eta)$ is skew-Hermitian
  for all $(t,y,\eta)\in [0,T]\times\R^d\times (\R^d\setminus0)$.
\end{enumerate}

We consider the Cauchy problem
\begin{equation}\label{ntd}
\left\{ \enspace
\begin{aligned}
  & \partial_t u + \cB(t,y,D_y)u = f(t,y), \quad (t,y)\in (0,T)\times\R^d, \\
  & u\bigr|_{t=0} = u_0(y).
\end{aligned}
\right.
\end{equation}

\begin{prop}\label{NTD}
Let $u_0\in H^{s+\sigma,\langle k\rangle}(\R^d;\C^N)$, $f\in
\bigcap_{r=0}^\sigma W^{r,1}((0,T);H^{s-r+\sigma,\langle
  k\rangle}(\R^d;\C^N))$ for some $(s,k)\in\R\times\Z$,
$\sigma\in\N_0$. Then \textup{Eq.~\eqref{ntd}} possesses a unique
solution
\[
  u\in \bigcap_{r=0}^\sigma\cC^r([0,T];H^{s-r+\sigma,\langle
    k\rangle}(\R^d;\C^N)).
\]
\end{prop}

\begin{proof}
We introduce the operator $M=\log^k \langle D_y\rangle$ and set $v = M
u$, $v_0= Mu_0$, and $g=Mf$. Then $v_0\in H^{s+\sigma}(\R^d;\C^N)$,
$g\in \bigcap_{r=0}^\sigma W^{r,1}((0,T);H^{s-r+\sigma}(\R^d;\C^N))$,
and $u\in \bigcap_{r=0}^\sigma\cC^r([0,T];$ $H^{s-r+\sigma,\langle
  k\rangle}(\R^d;\C^N))$ is equivalent to $v\in
\bigcap_{r=0}^\sigma\cC^r([0,T];H^{s-r+\sigma}(\R^d;\C^N))$. Furthermore,
$v$ solves the Cauchy problem
\begin{equation}\label{ntd1}
\left\{ \enspace
\begin{aligned}
  & \partial_t v + \left(\cB(t)+\left[M,\cB(t)\right]M^{-1}\right)v = g(t,y),
    \quad (t,y)\in (0,T)\times\R^d, \\
  & v\bigr|_{t=0} = v_0(y).
\end{aligned}
\right.
\end{equation}
Now, $\cB + [M,\cB]M^{-1}\in
\cC^\infty([0,T];\Psi^1(\R^d;\C^N)+\bigcap_{\epsilon>0}\Psi_{1,0}^\epsilon(\R^d;\C^N))$,
while $\sigma_\psi^1(\cB + \left[M,\cB\right]M^{-1}) =
\sigma_\psi^1(B)$. Then standard hyperbolic theory yields that
Eq.~\eqref{ntd1} possesses a unique solution $v\in
\bigcap_{r=0}^\sigma\cC^r([0,T];$ $H^{s-r+\sigma}(\R^d;\C^N))$.
\end{proof}




\nocite{*}

\bibliographystyle{abbrv}

\bibliography{tcbp}


\end{document}